\documentclass[preprint,10pt]{elsarticle}
\usepackage{amsfonts,amssymb,amsmath,empheq}
\usepackage{amsthm}
\usepackage{booktabs}
\usepackage{graphics}
\usepackage{graphicx}
\usepackage{latexsym}
\usepackage{subfigure}
\usepackage{multirow}
\usepackage{natbib}
\usepackage{float}
\usepackage{cases}
\usepackage{bm}
\usepackage{color}
\usepackage{indentfirst}
\usepackage{epstopdf}
\usepackage{CJK}
\usepackage[usenames,dvipsnames]{xcolor}
\usepackage[colorlinks,
            linkcolor=red,
            anchorcolor=blue,
            citecolor=green
            ]{hyperref}
\newtheorem{theorem}{Theorem}[section]

\newtheorem{lemma}{Lemma}[section]

\newcommand{\beqa}{\begin{eqnarray}}
\newcommand{\eeqa}{\end{eqnarray}}

\newcommand{\beq}{\begin{equation}}
\newcommand{\eeq}{\end{equation}}
\numberwithin{equation}{section}

\journal{}

\begin{document}

\begin{frontmatter}



\title{ Two regularized energy-preserving finite difference methods for the logarithmic Klein-Gordon equation }


\author[nudt,nus]{Jingye Yan}
\ead{yanjingye0205@163.com}

\author[nudt]{Xu Qian}
\ead{qianxu@nudt.edu.cn}
\author[nudt]{Hong Zhang\corref{cor1}}
\ead{zhanghnudt@163.com}
\author[nudt,hpc]{Songhe Song}
\ead{shsong@nudt.edu.cn}

\cortext[cor1]{Corresponding author.}

\address[nudt]{Department of Mathematics, College of Arts and Sciences, National University of Defense Technology, Changsha 410073, China}
\address[nus]{Department of Mathematics, National University of Singapore, Singapore 119076}
\address[hpc]{State Key Laboratory of High Performance Computing, National University of Defense Technology, Changsha 410073, China}

\begin{abstract}
We present and analyze two regularized finite difference methods which preserve energy of the logarithmic Klein-Gordon equation (LogKGE). In order to avoid singularity caused by the logarithmic nonlinearity of the LogKGE, we propose a regularized logarithmic Klein-Gordon equation (RLogKGE) with a small regulation parameter $0<\varepsilon\ll1$ to approximate the LogKGE with the convergence order $O(\varepsilon)$. By adopting the energy method, the inverse inequality, and the cut-off technique of the nonlinearity to bound the numerical solution, the error bound $O(h^{2}+\frac{\tau^{2}}{\varepsilon^{2}})$  of the two schemes with the mesh size $h$,  the time step $\tau$ and the parameter $\varepsilon$. Numerical results are reported to support our conclusions.
\end{abstract}

\end{frontmatter}

\section{Introduction}
\label{sec;introduction}

Consider the Klein-Gordon equation with the logarithmic nonlinear term (LogKGE),
\begin{align}\label{LogKGE}
\left\{
\begin{aligned}
&u_{tt}(\mathbf{x},t)-\Delta u(\mathbf{x},t)+u(\mathbf{x},t)+\lambda u(\mathbf{x},t)\ln(|u(\mathbf{x},t)|^{2})=0, ~~\mathbf{x}\in\Omega, t>0,\\
&u(\mathbf{x},0)=\phi(\mathbf{x}),~~~\partial_{t}u(\mathbf{x},0)=\gamma(\mathbf{x}),~~~\mathbf{x}\in \overline{\Omega},
\end{aligned}
\right.
\end{align}
where $\mathbf{x}=(x_{1},\ldots,x_{d})^{\mathrm{T}}\in \mathbb{R}^{d}, (d =1, 2, 3)$ is the spatial coordinate, $t$ is time, $u:=u(\mathbf{x},t)$ is a real-valued scalar field, the parameter $\lambda$ measures the force of the nonlinear interaction and $\Omega=\mathbb{R}^{d}$ or $\Omega\subset\mathbb{R}^{d}$ is a bounded domain with homogeneous Dirichlet or periodic boundary condition fixed on the boundary.

The LogKGE (\ref{LogKGE}) is a relativistic version of the logarithmic Schr\"{o}dinger equation \cite{bartkowski2008one} which was introduced by Rosen \cite{rosen1969dilatation} in the quantum field theory. Such kinds of nonlinearity arise from different applications, such as optics \cite{buljan2003incoherent}, nuclear physics \cite{de2003logarithmic,hefter1985application} ,
supersymmetric field theories \cite{rosen1969dilatation}, inflation cosmology \cite{enqvist1998q,linde1992strings}, and geophysics \cite{krolikowski2000unified}
to describe the spinless particle \cite{sakurai1967advanced}.

Similar to the nonlinear Klein-Gordon equation (NKGE),
\begin{align}\label{KGE}
\left\{
\begin{aligned}
&u_{tt}(\mathbf{x},t)-\Delta u(\mathbf{x},t)+u(\mathbf{x},t)+ u(\mathbf{x},t)^{3}=0, ~~\mathbf{x}\in\Omega, t>0,\\
&u(\mathbf{x},0)=\phi(\mathbf{x}),~~~\partial_{t}u(\mathbf{x},0)=\gamma(\mathbf{x}),~~~\mathbf{x}\in\overline{\Omega},
\end{aligned}
\right.
\end{align}
 the LogKGE (\ref{LogKGE}) admits the energy conservation law \cite{masmoudi2002nonlinear,machihara2001nonrelativistic}, which is defined as:
\begin{align}
E(t) =\int_{\Omega}\left[ (u_{t}(\mathbf{x},t))^{2}+(\nabla u(\mathbf{x},t))^{2}+(1-\lambda)u^{2}(\mathbf{x},t)+\lambda u^{2}(\mathbf{x},t) \ln \left(|u(\mathbf{x},t)|^{2} \right) \right] \mathrm{d} \mathbf{x} \equiv E(0), \label{energy}
\end{align}
where $u(\cdot,t)\in H^{1}(\Omega)$ and $\partial_{t}u(\cdot,t)\in L^{2}(\Omega)$.

The LogKGE (\ref{LogKGE}) has been studied theoretically in the literature, such as the existence of some special analytical solutions in quantum mechanics \cite{maslov1990pulsons,bialynicki1979gaussons,koutvitsky2006instability}, and the existence of classical solutions and weak solutions \cite{bartkowski2008one,gorka2009logarithmic}. In \cite{bialynicki1979gaussons}, the author studied the Gaussian solutions \cite{wazwaz2016gaussian}. Besides, the interaction of Gaussons has been investigated in \cite{makhankov1981interaction}.

For the numerical part, different efficient and accurate numerical methods have been proposed and analyzed for computations of wave propagations in classic/relativistic physics, such as the standard finite difference time domain (FDTD) \cite{bao2012analysis,chang1991conservative,duncan1997sympletic,bao2019long,luo2017fourth,zhang2005convergence,bao2013optimal,yan2020regularized}
, the conservative compact finite difference method \cite{ji2019conservative}, the multiscale time integrator Fourier pseudospectral (MWI-FP) method \cite{bao2014uniformly}, the finite element method \cite{cao1993fourier}, the time-splitting spectral method (TSFP) \cite{bao2020uniform}, the exponential wave integrator Fourier pseudospectral (EWI-FP) method \cite{bao2012analysis,bao2013exponential}, the asymptotic preserving (AP)  method \cite{faou2014asymptotic}, etc. Of course, each method has its advantages and disadvantages. For numerical comparisons of different numerical methods, we refer to \cite{bao2012analysis,jimenez1990analysis,pascual1995numerical,bao2019comparison}.

In the last thirty years, the conservation of invariants has draw much attention from different research fields \cite{feynman1965feynman,benjamin1972stability}. Various of conservative numerical methods have been proposed in the literature including the Crank-Nicolson finite difference (CNFD) and the semi-implicit finite difference (SIFD) \cite{bao2012uniform,bao2013optimal}, the average vector field (AVF) method \cite{celledoni2012preserving}, the local discontinuous Galerkin methods \cite{castillo2019conservative}, the local structure-preserving method \cite{cai2013local}, the
Hamiltonian boundary value method (HBVM) \cite{brugnano2018class},  the Kahan method \cite{celledoni2012geometric}, etc. It is a natural question to ask whether one can design numerical methods for the LogKGE (\ref{LogKGE}). However, these methods can not be applied and analyzed to the LogKGE (\ref{LogKGE}) directly, because of the singularity at the origin of the logarithmic term.
To our best knowledge, in \cite{yan2020regularized} the authors proposed and analyzed two regularized finite difference methods for the LogKGE (\ref{LogKGE}), but those schemes can not preserve energy. The main objective of this paper is to carry out and analyze two energy-preserving regularized finite difference methods for the LogKGE. In our numerical analysis, besides the standard techniques of the energy method, we employ the cut-off technique for dealing with the nonlinear term, and the inverse inequality for obtaining a posterior bound of the numerical solution.

The rest of this paper is organized as follows. In Section 2, a regularized version of LogKGE (\ref{LogKGE}) with a small parameter $0<\varepsilon \ll 1$ is proposed and we analyze the convergence of the energy between the LogKGE (\ref{LogKGE}) and the regularized logarithmic  Klein-Gordon equation (RLogKGE) (\ref{RLogKGE}).  In Section 3, two energy-preserving finite difference methods are proposed for the RLogKGE, and their
 properties of the stability, energy convergence, solvability are also analyzed. Section 4 is devoted to establishing the details of error bounds of the two numerical methods.  Numerical results are reported in Section 5 to confirm our theoretical analysis. Finally, some conclusions are drawn in Section 6. Throughout this paper, we let $p\lesssim q$ to denote that there exists a generic
constant $C$ which is independent of $\tau,h,\varepsilon$, such that $|p|\leq Cq$.
\section{A regularized logarithmic  Klein-Gordon equation}
In order to avoid blow-up and to suppress round-off error of the LogKGE (\ref{LogKGE}), a regularized logarithmic  Klein-Gordon equation (RLogKGE) with a small regulation parameter $0<\varepsilon \ll 1$ was introduced in \cite{yan2020regularized} as,
\begin{align}\label{RLogKGE}
\left\{
\begin{aligned}
&u^{\varepsilon}_{tt}(\mathbf{x},t)-\Delta u^{\varepsilon}(\mathbf{x},t)+u^{\varepsilon}(\mathbf{x},t)+\lambda u^{\varepsilon}(\mathbf{x},t)\ln\left(\varepsilon^{2}+ \left( u^{\varepsilon}(\mathbf{x},t) \right)^{2} \right)=0, ~~\mathbf{x}\in\Omega, t>0,\\
&u^{\varepsilon}(\mathbf{x},0)=\phi(\mathbf{x}),~~~\partial_{t}u^{\varepsilon}(\mathbf{x},0)=\gamma(\mathbf{x}),~~~\mathbf{x}\in\overline{\Omega}.
\end{aligned}
\right.
\end{align}
The above RLogKGE (\ref{RLogKGE}) is time symmetric and time reversible, i.e., they are invarient if interchanging $n+1\leftrightarrow n-1$ and $\tau \leftrightarrow -\tau$.

\begin{remark}
The regularization technique of the logarithmic term has been extensively investigated in \cite{bao2019error}. For the Cauchy problem of the LogKGE (\ref{LogKGE}) and the RLogKGE (\ref{RLogKGE}), the theoretical convergence estimate will be presented in our future work.
\end{remark}
\begin{theorem}
Assume $u^{\varepsilon}(\cdot,t)\in H^{1}(\Omega)$ and $\partial_{t}u^{\varepsilon}(\cdot,t)\in L^{2}(\Omega)$,
the RLogKGE (\ref{RLogKGE}) conserves the energy
, which is defined as:
\begin{align}
E^{\varepsilon}(t) =\int_{\Omega}\left[ (u^{\varepsilon}_{t}(\mathbf{x},t))^{2}+(\nabla u^{\varepsilon}(\mathbf{x},t))^{2}+(u^{\varepsilon}(\mathbf{x},t))^{2}+\lambda F_{\varepsilon} \left( (u^{\varepsilon}(\mathbf{x},t))^{2} \right) \right] \mathrm{d} \mathbf{x} \equiv E^{\varepsilon}(0), \label{energyr}
\end{align}
where $F_{\varepsilon}(\rho)=\int_{0}^{\rho} \ln (\varepsilon^{2}+s) \mathrm{d} s= \rho\ln (\varepsilon^{2}+\rho)+\varepsilon^{2}\ln (1+\frac{\rho}{\varepsilon^{2}})- \rho.$
\end{theorem}
\begin{proof}
\begin{align}
\begin{split}
\frac{\mathrm{d}}{\mathrm{d}t}E^{\varepsilon}(t)&=2\int_{\Omega}\left[ u^{\varepsilon}_{t}\cdot u^{\varepsilon}_{tt}+\nabla u^{\varepsilon}\cdot \nabla u^{\varepsilon}_{t}+u^{\varepsilon}u^{\varepsilon}_{t}+\lambda F^{\prime}_{\varepsilon}\left((u^{\varepsilon})^{2}\right)\cdot u^{\varepsilon}\cdot u^{\varepsilon}_{t}\right](\mathbf{x},t) \mathrm{d} \mathbf{x}\\
&=2\int_{\Omega} \left[ u^{\varepsilon}_{t} \left( u^{\varepsilon}_{tt}-\Delta u^{\varepsilon}+u^{\varepsilon}+\lambda u^{\varepsilon}\ln\left(\varepsilon^{2}+( u^{\varepsilon})^{2} \right) \right) \right] (\mathbf{x},t) \mathrm{d} \mathbf{x}=0.
\end{split}
\end{align}
This ends the proof.
\end{proof}

Since the above regularized energy involves the $L^{1}$-norm of $u^{\varepsilon}$ for any $\varepsilon$, only when $\Omega$ has finite measure, not when $\Omega=\mathbb{R}^{d}$, $E^{\varepsilon}$ is obviously well-defined for $u_{0}\in H^{1}(\Omega)$.
\subsection{Convergence of the energy}
Here, we will present the convergence of the energy $E^{\varepsilon}(u_{0})\rightarrow E(u_{0})$.
\begin{theorem}
For $u_{0}\in H^{1}(\Omega)\cap L^{1}(\Omega)$, the energy $E^{\varepsilon}(u_{0})$ converges to $E(u_{0})$ with
\begin{align*}
\left|E^{\varepsilon}\left(u_{0}\right)-E\left(u_{0}\right)\right|\leq 4\varepsilon|\lambda|\|u_{0}\|_{L^{1}(\Omega)}.
\end{align*}
\end{theorem}
\begin{proof}
\begin{align*}
\left|E^{\varepsilon}\left(u_{0}\right)-E\left(u_{0}\right)\right|&=|\lambda|\cdot \bigg|\int_{\Omega}\left[ u_{0}^{2}\ln(\varepsilon^{2}+u_{0}^{2}) +\varepsilon^{2}\ln\left(1+\frac{u^{2}_{0}}{\varepsilon^{2}}\right)-u^{2}_{0}\ln(u_{0}^{2})\right] \mathrm{d}\mathbf{x}\bigg|\\
&= |\lambda|\cdot \bigg|\int_{\Omega}\left[ u_{0}^{2}\ln\left(\frac{\varepsilon^{2}+u_{0}^{2}}{u_{0}^{2}}\right) +\varepsilon^{2}\ln\left(1+\frac{u^{2}_{0}}{\varepsilon^{2}}\right)\right] \mathrm{d}\mathbf{x}\bigg|\\
&\leq |\lambda|\cdot \bigg|\int_{\Omega}\left[ u_{0}^{2}\ln\left(\frac{\varepsilon^{2}+u_{0}^{2}+2\varepsilon|u_{0}|}{u_{0}^{2}}\right) +\varepsilon^{2}\ln\left(\frac{\varepsilon^{2}+u_{0}^{2}+2\varepsilon|u_{0}|}{\varepsilon^{2}}\right)\right] \mathrm{d}\mathbf{x}\bigg|\\
&= |\lambda|\cdot \bigg|\int_{\Omega}\left[ u_{0}^{2}\ln\left(\frac{(\varepsilon+|u_{0}|)^{2}}{u_{0}^{2}}\right) +\varepsilon^{2}\ln\left(\frac{(\varepsilon+|u_{0}|)^{2}}{\varepsilon^{2}}\right)\right] \mathrm{d}\mathbf{x}\bigg|\\
&= |\lambda|\cdot \bigg|\int_{\Omega}2\left[ u_{0}^{2}\ln\left(1+\frac{\varepsilon}{|u_{0}|}\right) +\varepsilon^{2}\ln\left(1+\frac{|u_{0}|}{\varepsilon}\right)\right] \mathrm{d}\mathbf{x}\bigg|\\
&\leq 4\varepsilon|\lambda|\|u_{0}\|_{L^{1}(\Omega)},
\end{align*}
according to the inequality $0\leq \ln(1+|x|)\leq |x|$.

This ends the proof.
\end{proof}

\section{The FDTD methods and their analysis}
In this section, we construct two FDTD schemes for the RLogKGE (\ref{RLogKGE}) and study their properties, such as stability, energy conservation and solvability. For simplicity, we set $\lambda=1$ and only present numerical schemes and theoretical analysis in one dimensional space ($d=1$). In practical computation, we truncate the whole space problem onto an interval $\Omega=(a,b)$ with periodic boundary conditions
\begin{align}\label{RLogKGE1}
\left\{
\begin{aligned}
&u^{\varepsilon}_{tt}(x,t)-\Delta u^{\varepsilon}(x,t)+u^{\varepsilon}(x,t)+ u^{\varepsilon}(x,t)\ln\left(\varepsilon^{2}+ \left( u^{\varepsilon}(x,t) \right)^{2} \right)=0, ~~x\in\Omega=(a,b), ~~t>0,\\
&u^{\varepsilon}(x,0)=\phi(x),~~~\partial_{t}u^{\varepsilon}(x,0)=\gamma(x),~~~x\in\overline{\Omega}=[a,b].
\end{aligned}
\right.
\end{align}

\subsection{The FDTD methods}

Choose the  time step $\tau:=\Delta t$ and the mesh size $h:=\frac{b-a}{N}$ with $N$ being a positive integer, denote the time steps as $t_{n}:=n\tau, n=0,1,2,\ldots;$  the grid points as $x_{j}:=a+jh,j=0,1,\ldots,N$, and define the index sets: $\mathcal{T}_{N}^{0}=\left\{j|j=0,1,2,\ldots,N\right\}, \mathcal{T}_{N}=\left\{j|j=0,1,2,\ldots,N-1\right\}$. 

Assume $u^{\varepsilon,n}_{j}, u^{n}_{j}$ are the approximations of the exact solutions $u^{\varepsilon}(x_{j},t_{n})$ and $ u(x_{j},t_{n})$, $j\in \mathcal{T}_{N}^{0}$ and $n\geq 0$. Define $u^{\varepsilon,n}=(u^{\varepsilon,n}_{0},u^{\varepsilon,n}_{1},\ldots,u^{\varepsilon,n}_{N})^{\mathrm{T}}, u^{n}=(u^{n}_{0},u^{n}_{1},\ldots,u^{n}_{N})^{\mathrm{T}} \in \mathbb{R}^{N+1} $ as the numerical solution vectors at time $t=t_{n}$. The followings are the finite difference operators:
\begin{align*}
&\delta^{+}_{t}u^{n}_{j}=\frac{u^{n+1}_{j}-u^{n}_{j}}{\tau},~~~~~\delta_{t}^{-}u^{n}_{j}=\frac{u^{n}_{j}-u^{n-1}_{j}}{\tau},~~~~~\delta_{t}u^{n}_{j}=\frac{u^{n+1}_{j}-u^{n-1}_{j}}{2\tau},~~~~\delta^{2}_{t}u^{n}_{j}=\frac{u^{n+1}_{j}-2u^{n}_{j}+u^{n-1}_{j}}{\tau^{2}},\\
&\delta^{+}_{x}u^{n}_{j}=\frac{u^{n}_{j+1}-u^{n}_{j}}{h},~~~~~~\delta^{-}_{x}u^{n}_{j}=\frac{u^{n}_{j}-u^{n}_{j-1}}{h},~~~~~~\delta_{x}u^{n}_{j}=\frac{u^{n}_{j+1}-u^{n}_{j-1}}{2h},~~~~~\delta^{2}_{x}u^{n}_{j}=\frac{u^{n}_{j+1}-2u^{n}_{j}+u^{n}_{j-1}}{h^{2}}.
\end{align*}
We denote a space of grid functions
\begin{align}
X_{N}=\left\{ u| u=(u_{0},u_{1},u_{2},\ldots,u_{N})^{\mathrm{T}},u_{0}=u_{N}\right\}\subseteq \mathbb{R}^{N+1},
\end{align}
and we always use $u_{-1}=u_{N-1}$ and $u_{1}=u_{N+1}$ if they are involved.

We define the standard discrete $l^{2}$, semi-$H^{1}$ and $l^{\infty}$ norms and inner product over $X_{N}$ as follows
\begin{align}
\|u\|_{l^{2}}^{2}=(u,u)=h\sum \limits^{N-1}_{j=0}|u_{j}|^{2},~~\|\delta_{x}^{+}u\|^{2}_{l^{2}}=h\sum \limits^{N-1}_{j=0}|\delta_{x}^{+}u_{j}|^{2},~~\|u\|_{l^{\infty}}=\sup\limits_{0\leq j\leq N-1}|u_{j}|,~~(u,v)=h\sum \limits^{N-1}_{j=0}u_{j}v_{j},
\end{align}
where $u,v\in X_{N}$, and $(\delta^{2}_{x}u,v)=-(\delta^{+}_{x}u,\delta^{+}_{x}v)=(u,\delta^{2}_{x}v)$.
In the following, we introduce two frequently used FDTD methods for the RLogKGE (\ref{RLogKGE}):\\

$\mathbf{I.}$ The Crank-Nicolson finite difference (CNFD) method
\begin{align}\label{CNFD}
\delta_{t}^{2}u^{\varepsilon,n}_{j}-\frac{1}{2}\delta_{x}^{2} (u^{\varepsilon,n+1}_{j}+u^{\varepsilon,n-1}_{j})+\frac{1}{2} (u^{\varepsilon,n+1}_{j}+u^{\varepsilon,n-1}_{j})+ G_{\varepsilon}(u^{\varepsilon,n+1}_{j},u^{\varepsilon,n-1}_{j})=0,~~n\geq1;
\end{align}

$\mathbf{II.}$ A semi-implicit energy conservative finite difference (SIEFD) method
\begin{align}\label{SIFD1}
\delta_{t}^{2}u^{\varepsilon,n}_{j}-\delta_{x}^{2} u^{\varepsilon,n}_{j}+\frac{1}{2} (u^{\varepsilon,n+1}_{j}+u^{\varepsilon,n-1}_{j})+ G_{\varepsilon}(u^{\varepsilon,n+1}_{j},u^{\varepsilon,n-1}_{j})=0,~~n\geq1.
\end{align}
Here, $G_{\varepsilon}(z_{1},z_{2})$ is defined for $z_{1}, z_{2}\in \mathbb{R}$ as
\begin{align}\label{G}
\begin{split}
G_{\varepsilon}(z_{1},z_{2}):&=\int_{0}^{1} f_{\varepsilon}(\theta z_{1}^{2}+(1-\theta)z_{2}^{2})\mathrm{d} \theta\cdot \frac{z_{1}+z_{2}}{2}\\
&=\frac{F_{\varepsilon}(z_{1}^{2})-F_{\varepsilon}(z_{2}^{2})}{z_{1}^{2}-z_{2}^{2}}\cdot \frac{z_{1}+z_{2}}{2},
\end{split}
\end{align}
with
\begin{align}
&f_{\varepsilon}(\rho)=\ln (\varepsilon^{2}+\rho),\\
&F_{\varepsilon}(\rho)=\int_{0}^{\rho} f_{\varepsilon}(s)\mathrm{d} s=\rho\ln (\varepsilon^{2}+\rho)+\varepsilon^{2}\ln (1+\frac{\rho}{\varepsilon^{2}})- \rho,   \quad \rho\geq0.
\end{align}
The initial and boundary conditions are discretized as
\begin{align}\label{initialvalue1}
u^{\varepsilon,n+1}_{0}=u^{\varepsilon,n+1}_{N},u^{\varepsilon,n+1}_{-1}=u^{\varepsilon,n+1}_{N-1},~~n\geq0,~~u^{\varepsilon,0}_{j}=\phi(x_{j}),~~j\in \mathcal{T}_{N}^{0}.
\end{align}
Besides, according to the Taylor expansion we can approximate the first step solution $u_{j}^{\varepsilon,1}$ by,
\begin{align}\label{initialvalue}
u_{j}^{\varepsilon,1}=\phi(x_{j})+\tau\gamma(x_{j})+\frac{\tau^{2}}{2}\left[\delta^{2}_{x}\phi(x_{j})-\phi(x_{j})- \phi(x_{j})\ln(\varepsilon^{2} +(\phi(x_{j}))^{2}) \right], ~~j\in \mathcal{T}_{N}^{0}.
\end{align}
It is easy to prove that the above FDTD schemes are all time symmetric and time reversible, i.e. they are invarient if interchanging $n+1\leftrightarrow n-1$ and $\tau \leftrightarrow -\tau$.
\begin{lemma}\cite{bao2012analysis}\label{innerproduct}
For any $u^{n}\in X_{N}~(n\geq 0)$, we can obtain
\begin{subequations}
\begin{align}
&-h \sum_{j=0}^{N-1} u_{j}^{n} \delta_{x}^{2} u_{j}^{n}=h \sum_{j=0}^{N-1}\left|\delta_{x}^{+} u_{j}^{n}\right|^{2}=\left\|\delta_{x}^{+} u^{n}\right\|_{l^{2}}^{2},\\
&h \sum_{j=0}^{N-1} u_{j}^{n} u_{j}^{n+1}=\frac{1}{2}\left\|u^{n}\right\|_{l^{2}}^{2}+\frac{1}{2}\left\|u^{n+1}\right\|_{l^{2}}^{2}-\frac{\tau^{2}}{2}\left\|\delta_{t}^{+} u^{n}\right\|_{l^{2}}^{2}, \\
&h \sum_{j=0}^{N-1}\left(\delta_{x}^{+} u_{j}^{n+1}\right)\left(\delta_{x}^{+} u_{j}^{n}\right)=\frac{1}{2 h} \sum_{j=0}^{N-1}\left[\left(u_{j+1}^{n+1}-u_{j}^{n}\right)^{2}+\left(u_{j+1}^{n}-u_{j}^{n+1}\right)^{2}\right] \\
&~~~~~~~~~~~~~~~~~~~~~~~~~~~~~~~~-\frac{\tau^{2}}{h^{2}}\left\|\delta_{t}^{+} u^{n}\right\|_{l^{2}}^{2}, \quad n=0,1, \ldots\notag.
\end{align}
\end{subequations}
\end{lemma}
\begin{theorem}
The discrete scheme (\ref{CNFD}) satisfies the discrete energy conservation law:
\begin{align}\label{CNFDE}
\begin{split}
E^{\varepsilon,n}:=&\|\delta_{t}^{+}u^{\varepsilon,n}\|^{2}_{l^{2}}+\frac{1}{2}\left( \|\delta_{x}^{+}u^{\varepsilon,n+1}\|^{2}_{l^{2}}+ \|\delta_{x}^{+}u^{\varepsilon,n}\|^{2}_{l^{2}} \right)+\frac{1}{2}\left( \|u^{\varepsilon,n+1}\|^{2}_{l^{2}}+\|u^{\varepsilon,n}\|^{2}_{l^{2}} \right)\\
&+\frac{ h}{2}\sum \limits^{N-1}_{j=0}\left[ F_{\varepsilon}((u^{\varepsilon,n+1}_{j})^{2})+F_{\varepsilon}((u^{\varepsilon,n}_{j})^{2}) \right] \equiv E^{\varepsilon,0},~~~n=0,1,2,\ldots
\end{split}
\end{align}
Similarly, the scheme (\ref{SIFD1}) conserves :
\begin{align}\label{SIFD1E}
\begin{split}
\widetilde{E}^{\varepsilon,n}:=&\|\delta_{t}^{+}u^{\varepsilon,n}\|^{2}_{l^{2}}+h\sum \limits^{N-1}_{j=0}( \delta_{x}^{+}u^{\varepsilon,n+1}_{j}) (\delta_{x}^{+}u^{\varepsilon,n}_{j}) +\frac{1}{2}\left( \|u^{\varepsilon,n+1}\|^{2}_{l^{2}}+\|u^{\varepsilon,n}\|^{2}_{l^{2}} \right)\\
&+\frac{h}{2}\sum \limits^{N-1}_{j=0}\left[ F_{\varepsilon}((u^{\varepsilon,n+1}_{j})^{2})+F_{\varepsilon}((u^{\varepsilon,n}_{j})^{2}) \right] \equiv \widetilde{E}^{\varepsilon,0},~~~n=0,1,2,\ldots
\end{split}
\end{align}
\end{theorem}
\begin{proof}
By taking the inner of (\ref{CNFD}) with $\delta_{t}u^{n}_{j}$, summing the identity together for $j$ from $0$ to $N-1$, and using Lemma \ref{innerproduct}, we have
\begin{align}
\begin{split}
&\|\delta_{t}^{+}u^{\varepsilon,n}\|^{2}_{l^{2}}-\|\delta_{t}^{+}u^{\varepsilon,n-1}\|^{2}_{l^{2}}+\frac{1}{2}( \|\delta_{x}^{+}u^{\varepsilon,n+1}\|^{2}_{l^{2}}- \|\delta_{x}^{+}u^{\varepsilon,n-1}\|^{2}_{l^{2}}\\
&+\|u^{\varepsilon,n+1}\|^{2}_{l^{2}}-\|u^{\varepsilon,n-1}\|^{2}_{l^{2}})+\sum_{j=0}^{N-1}\frac{h}{2}\left[F_{\varepsilon}((u^{\varepsilon,n+1}_{j})^{2})-F_{\varepsilon}((u^{\varepsilon,n-1}_{j})^{2})\right]=0. \end{split}
\end{align}
Define
\begin{align}
\begin{split}
E^{\varepsilon,n}:=&\|\delta_{t}^{+}u^{\varepsilon,n}\|^{2}_{l^{2}}+\frac{1}{2}\left( \|\delta_{x}^{+}u^{\varepsilon,n+1}\|^{2}_{l^{2}}+ \|\delta_{x}^{+}u^{\varepsilon,n}\|^{2}_{l^{2}} \right)+\frac{1}{2}\left( \|u^{\varepsilon,n+1}\|^{2}_{l^{2}}+\|u^{\varepsilon,n}\|^{2}_{l^{2}} \right)\\
&+\frac{ h}{2}\sum \limits^{N-1}_{j=0}\left[ F_{\varepsilon}((u^{\varepsilon,n+1}_{j})^{2})+F_{\varepsilon}((u^{\varepsilon,n}_{j})^{2}) \right],~~~n=0,1,2,\ldots,
\end{split}
\end{align}
we can obtain
\begin{align}
{E}^{\varepsilon,n}\equiv {E}^{\varepsilon,0}.
\end{align}
Similarly, we can get
\begin{align}
\begin{split}
\widetilde{E}^{\varepsilon,n}:=&\|\delta_{t}^{+}u^{\varepsilon,n}\|^{2}_{l^{2}}+h\sum \limits^{N-1}_{j=0}( \delta_{x}^{+}u^{\varepsilon,n+1}_{j}) (\delta_{x}^{+}u^{\varepsilon,n}_{j}) +\frac{1}{2}\left( \|u^{\varepsilon,n+1}\|^{2}_{l^{2}}+\|u^{\varepsilon,n}\|^{2}_{l^{2}} \right)\\
&+\frac{h}{2}\sum \limits^{N-1}_{j=0}\left[ F_{\varepsilon}((u^{\varepsilon,n+1}_{j})^{2})+F_{\varepsilon}((u^{\varepsilon,n}_{j})^{2}) \right] \equiv \widetilde{E}^{\varepsilon,0},~~~n=0,1,2,\ldots.
\end{split}
\end{align}
This ends the proof.
\end{proof}

\subsection{Stability analysis}
Let $0<T<T_{\max}$ with $T_{\max}$ being the maximum existence time. Define
\begin{align}
\sigma_{\max}:=\max\{|\ln(\varepsilon^{2})|,|\ln(\varepsilon^{2}+\|u^{\varepsilon,n}\|^{2}_{l^{\infty}})|\},~~0\leq n\leq\frac{T}{\tau}-1.
\end{align}
According to the von Neumann linear stability analysis, we can get the following stability results for the FDTD schemes.
\begin{theorem}\label{stability}
For the above FDTD schemes applied to the RLogKGE (\ref{RLogKGE}) up to $t=T$, we can obtain:

(i)The CNFD scheme (\ref{CNFD}) is unconditionally stable for any $h>0,\tau>0, 0< \varepsilon \ll 1$.

(ii)When $4-h^{2}(1+\sigma_{\max})\leq 0$, the SIEFD scheme (\ref{SIFD1}) is unconditionally stable; and when $4-h^{2}(1+\sigma_{\max})>0$, it is conditionally stable under the stability condition
\begin{align}\label{sifd1stability}
\tau\leq \frac{2h}{\sqrt{4-h^{2}-h^{2}\sigma_{\max}}}.
\end{align}
\end{theorem}
\begin{proof}
Substituting
\begin{align}
u^{\varepsilon,n-1}_{j}=\sum_{l}\hat{U}_{l}e^{2ijl\pi/N},~~u^{\varepsilon,n}_{j}=\sum_{l}\xi_{l}\hat{U}_{l}e^{2ijl\pi/N},~~u^{\varepsilon,n+1}_{j}=\sum_{l}\xi_{l}^{2}\hat{U}_{l}e^{2ijl\pi/N},~~
\end{align}
into (\ref{CNFD})-(\ref{SIFD1}), where $\xi_{l}$ is the amplification factor of the $l$th mode in phase space, we can get the characteristic equation with the following structure
\begin{align}
\xi_{l}^{2}-2 \theta_{l} \xi_{l}+1=0, \quad l=-\frac{N}{2}, \ldots, \frac{N}{2}-1,
\end{align}
where $\theta_{l}$ is invariant with different methods. By the above equation, we get $\xi_{l}=\theta_{l}\pm \sqrt{\theta_{l}^{2}-1}$. The stability of numerical schemes amounts to
\begin{align}
\left|\xi_{l}\right| \leq 1 \Longleftrightarrow\left|\theta_{l}\right| \leq 1, \quad l=-\frac{N}{2}, \ldots, \frac{N}{2}-1.
\end{align}
Denote $s_{l}=\frac{2}{h} \sin \left(\frac{l \pi}{N}\right), \quad l=-\frac{N}{2}, \ldots, \frac{N}{2}-1$, we have
\begin{align}\label{stability1}
0\leq s_{l}^{2}\leq \frac{4}{h^{2}}.
\end{align}
Firstly, we consider the situation of linearity, i.e. $f_{\varepsilon}=\alpha$, $\alpha$ is a constant satisfying $\alpha\geq-1$.

(i) For the CNFD scheme (\ref{CNFD}), we have
\begin{align}
0 \leq \theta_{l}=\frac{2 }{2 +\tau^{2}\left(\alpha+s_{l}^{2}+1\right)} \leq 1, \quad l=-\frac{N}{2}, \ldots, \frac{N}{2}-1.
\end{align}
We can conclude that the CNFD scheme (\ref{CNFD}) is unconditionally stable for any $h>0,\tau>0, 0< \varepsilon \ll 1$.

(ii) For the SIEFD scheme (\ref{SIFD1}), we have
\begin{align}
\theta_{l}=\frac{2- s_{l}^{2}\tau^{2}}{2 +\tau^{2}\left(\alpha+1\right)} , \quad l=-\frac{N}{2}, \ldots, \frac{N}{2}-1.
\end{align}
When $4-h^{2}(1+\alpha)\leq 0$, it implies that $|\theta_{l}| \leq 1$ and the SIEFD scheme (\ref{SIFD1}) is unconditionally stable. On the other hand, when $4-h^{2}(1+\alpha)>0$, under the condition $\tau\leq \frac{2h}{\sqrt{4-h^{2}-h^{2}\alpha}}$, we have
\begin{align}
(s_{l}^{2}-1-\alpha)\tau^{2}\leq (\frac{4}{h^{2}}-1-\alpha)\tau^{2}\leq 4,
\end{align}
it implies that, when $\tau\leq \frac{2h}{\sqrt{4-h^{2}-h^{2}\alpha}}$, the SIEFD scheme (\ref{SIFD1}) is stable.

Similarly, when $f_{\varepsilon}$ is nonlinear, we can get the conclusions of Theorem \ref{stability}.
\end{proof}
\subsection{Solvability and conservation}
\begin{lemma}\label{CNFDsol}(solvability of the CNFD) For any given $u^{\varepsilon,n},u^{\varepsilon,n-1},\phi^{\varepsilon,n}\in X_{N},~ (n\geq1)$, denote $C^{n}=f_{\varepsilon}((\phi^{\varepsilon,n})^{2})=\ln (\varepsilon^{2}+(\phi^{\varepsilon,n})^{2})$, there exists $\tau_{s}>0$ which depends on $C^{n}$ such that when $\tau<\tau_{s}$, the solution $u^{\varepsilon,n+1}$ of the CNFD (\ref{CNFD}) is unique at each time.
\end{lemma}
\begin{proof}
Firstly, we prove the solution existence of the CNFD (\ref{CNFD}). We denote $\hat{u}^{\varepsilon,n}_{j}=\frac{u^{\varepsilon,n+1}_{j}+u^{\varepsilon,n-1}_{j}}{2},~j\in \mathcal{T}^{0}_{N},~n\geq 1,~\hat{u}^{\varepsilon,n}\in X_{N}$. For any given $u^{\varepsilon,n+1}, u^{\varepsilon,n}, u^{\varepsilon,n-1} \in X_{N}$, we rewrite the CNFD as
\begin{align}
\hat{u}^{\varepsilon,n}=u^{\varepsilon,n}+\frac{\tau^{2}}{2}M^{n}_{\varepsilon}(\hat{u}^{\varepsilon,n}),~~n\geq 1,
\end{align}
where $M^{n}_{\varepsilon}: X_{N}\rightarrow X_{N}$ is defined as
\begin{align}
M^{n}_{\varepsilon,j}(v)=\delta_{x}^{2}v_{j}-v_{j}-H_{\varepsilon}^{n}(v_{j})\cdot v_{j},~j\in \mathcal{T}^{0}_{N},~n\geq 1,
\end{align}
with $H_{\varepsilon}^{n}(v_{j})=\frac{F_{\varepsilon}((2v_{j}-u^{\varepsilon,n-1}_{j})^{2})-F_{\varepsilon}((u^{\varepsilon,n-1}_{j})^{2})}{(2v_{j}-u^{\varepsilon,n-1}_{j})^{2}-(u^{\varepsilon,n-1}_{j})^{2}}$.
There exists a $\phi^{\varepsilon,n} \in X_{N}$ satisfying
\begin{align}
\frac{F_{\varepsilon}((2v_{j}-u^{\varepsilon,n-1}_{j})^{2})-F_{\varepsilon}((u^{\varepsilon,n-1}_{j})^{2})}{(2v_{j}-u^{\varepsilon,n-1}_{j})^{2}-(u^{\varepsilon,n-1}_{j})^{2}}=f_{\varepsilon}((\phi^{\varepsilon,n})^{2})=C^{n}.
\end{align}
According to the Cauchy inequality, the Sobolev inequality and the Young's inequality, we have
\begin{align}\label{uniq1}
\begin{split}
&|(H_{\varepsilon}^{n}(v)\cdot v,v)|=|(H_{\varepsilon}^{n}(v),v^{2})|\leq \|v\|^{2}_{4}\|H_{\varepsilon}^{n}\|_{l^{2}}\leq C\|\delta^{+}_{x}v\|_{l^{2}}^{\frac{1}{2}}\|v\|^{\frac{3}{2}}_{l^{2}}\|H_{\varepsilon}^{n}(v)\|_{l^{2}}\\
&\leq CC^{n}\|\delta^{+}_{x}v\|^{\frac{1}{2}}_{l^{2}}\|v\|^{\frac{3}{2}}_{l^{2}}\leq\|\delta^{+}_{x}v\|^{2}_{l^{2}}+(CC^{n})^{\frac{4}{3}}\|v\|^{2}_{l^{2}}.
\end{split}
\end{align}

Define the map $K^{n}_{\varepsilon}: X_{N}\rightarrow X_{N}$ as
\begin{align}
K^{n}_{\varepsilon}(v)=v-u^{\varepsilon,n}-\frac{\tau^{2}}{2}M^{n}_{\varepsilon}(v), ~v\in X_{N},~n\geq 1.
\end{align}
We can see that $K^{n}_{\varepsilon}$ is continuous from $X_{N}$ to $X_{N}$. In addition,
\begin{align}
\begin{split}
(K^{n}_{\varepsilon}(v),v)&=\|v\|^{2}_{l^{2}}-(u^{\varepsilon,n},v)+\frac{\tau^{2}}{2}\left[\|\delta_{x}^{+}v\|^{2}_{l^{2}}+ \|v\|^{2}_{l^{2}}+(H_{\varepsilon}^{n}(v_{j}), v^{2})\right]\\
&\geq \|v\|^{2}_{l^{2}}-\|u^{\varepsilon,n}\|_{l^{2}}\cdot\|v\|_{l^{2}}+\frac{\tau^{2}}{2}\|\delta_{x}^{+}v\|^{2}_{l^{2}}-\frac{\tau^{2}}{2}\|\delta_{x}^{+}v\|^{2}_{l^{2}}-\frac{\tau^{2}}{2}(C^{n}C)^{\frac{4}{3}}\cdot \|v\|^{2}_{l^{2}}\\
&=\left(\left(1-\frac{\tau^{2}}{2}(C^{n}C)^{\frac{4}{3}}\right)\|v\|_{l^{2}}-\|u^{n}\|_{l^{2}} \right)\|v\|_{l^{2}}.
\end{split}
\end{align}
Let $\tau_{s}=(C^{n}C)^{-\frac{2}{3}}$, when $\tau<\tau_{s}$, it is obviously that
\begin{align}
\lim_{\|v\|_{l^{2}}\rightarrow \infty}\frac{\left|(K^{n}_{\varepsilon}(v),v)\right|}{\|v\|_{l^{2}}}=\infty,~~n\geq1.
\end{align}
So $K^{n}_{\varepsilon}(v)$ is surjective. According to the Brouwer fixed point theorem \cite{landes1980galerkin}, it is easy to show that there exists a solution $u^{*}$ satisfying $K^{n}_{\varepsilon}(u^{*})=0$, which implies that the CNFD (\ref{CNFD}) is solvable.

Next, we verify the uniqueness of the solution of the CNFD (\ref{CNFD}).
We assume there exist two solutions $p,q\in X_{N}$ satisfying (\ref{CNFD}), and $m=p-q \in X_{N}$. It implies that
\begin{subequations}\label{CNFDuniq}
\begin{align}
&\frac{p_{j}-2u^{\varepsilon,n}_{j}+u^{\varepsilon,n-1}_{j}}{\tau^{2}}-\frac{1}{2}\delta_{x}^{2} (p_{j}+u^{\varepsilon,n-1}_{j})+\frac{1}{2} (p_{j}+u^{\varepsilon,n-1}_{j})+H_{\varepsilon}^{n}(p_{j})\cdot\frac{p_{j}+u^{\varepsilon,n-1}_{j}}{2}=0,\label{CNFDp}\\
&\frac{q_{j}-2u^{\varepsilon,n}_{j}+u^{\varepsilon,n-1}_{j}}{\tau^{2}}-\frac{1}{2}\delta_{x}^{2} (q_{j}+u^{\varepsilon,n-1}_{j})+\frac{1}{2} (q_{j}+u^{\varepsilon,n-1}_{j})+H_{\varepsilon}^{n}(q_{j})\cdot\frac{q_{j}+u^{\varepsilon,n-1}_{j}}{2}=0.\label{CNFDq}
\end{align}
\end{subequations}
Subtracting (\ref{CNFDq}) from (\ref{CNFDp}), we obtain
\begin{align}
\frac{m_{j}}{\tau^{2}}-\frac{1}{2}\delta_{x}^{2} m_{j}+\frac{1}{2}m_{j}+H_{\varepsilon}^{n}(p_{j})\cdot\frac{m_{j}}{2}+\big(H_{\varepsilon}^{n}(p_{j})-H_{\varepsilon}^{n}(q_{j})\big)\cdot\frac{q_{j}+u^{\varepsilon,n-1}_{j}}{2}=0,
\end{align}
which yields that
\begin{align}
\|m\|^{2}_{l^{2}}+\frac{\tau^{2}}{2}\left(\|\delta_{x}^{+} m\|^{2}_{l^{2}}+\|m\|^{2}_{l^{2}}\right)=-\frac{\tau^{2}}{2} \bigg{(} H_{\varepsilon}^{n}(p)\cdot m+\big{(}H_{\varepsilon}^{n}(p)-H_{\varepsilon}^{n}(q)\big{)}\cdot(q+u^{\varepsilon,n-1}),m \bigg{)}.
\end{align}
By recalling (\ref{uniq1}), when $\tau< \tau_{s}$, we get
\begin{align}
\|m\|^{2}_{l^{2}}+\frac{\tau^{2}}{2}(\|\delta_{x}^{+} m\|^{2}_{l^{2}}+\|m\|^{2}_{l^{2}})\leq \frac{\tau^{2}}{2} \left( \|\delta_{x}^{+}m\|^{2}_{l^{2}}+(C^{n}C)^{\frac{4}{3}}\cdot \|m\|^{2}_{l^{2}}+C\right)\leq \frac{\tau^{2}}{2} \|\delta_{x}^{+}m\|^{2}_{l^{2}}+\|m\|^{2}_{l^{2}},
\end{align}
which implies $m=0$. This ends the proof.
\end{proof}
\begin{remark}
The solvability of the SIEFD (\ref{SIFD1}) could be similarly proved using the same approach, therefore, we omit it.
\end{remark}

\section{Error esitimates}
Motivated by the analytical estimates in \cite{bao2013optimal,bao2012uniform}, we will establish the error estimates of the FDTD schemes. Here we assume the exact solution $u^{\varepsilon}$ is smooth enough over $\Omega_{T}:\Omega\times [0,T],$ i.e.
\begin{align*}
(A)~~~&u^{\varepsilon}\in C\left( [0,T]; W^{5,\infty}_{p}\right)\cap C^{2}\left( [0,T]; W^{3,\infty}_{p}\right)\cap C^{4}\left( [0,T]; W^{1,\infty}_{p}\right),\\
~~&\left\|\frac{\partial^{r+s}}{\partial t^{r}\partial x^{s}}u^{\varepsilon}(x,t) \right\|_{L^{\infty}(\Omega_{T})}\lesssim 1,~~0\leq r\leq 4, ~0\leq r+s\leq 5,
\end{align*}
where $W_{p}^{m, \infty}=\left\{u^{\varepsilon}\in W^{m, \infty} | \frac{\partial^{l}}{\partial x^{l}}u^{\varepsilon}(a)= \frac{\partial^{l}}{\partial x^{l}}u^{\varepsilon}(b), 0 \leq l \leq m-1\right\}$ for $m \geq 1$
 and $0<T<T^{*}$ with $T^{*}$ being the maximum existence time of the solution.

Denote $\Lambda=\|u^{\varepsilon}(x,t)\|_{L^{\infty}(\Omega_{T})}$ and the grid `error' function $e^{\varepsilon,n} \in X_{N}~(n\geq0)$ as
\begin{align}
e^{\varepsilon,n}_{j}=u^{\varepsilon}(x_{j},t_{n})-u^{\varepsilon,n}_{j},~~j\in \mathcal{T}^{0}_{N},~~n=0,1,2,\ldots,
\end{align}
where $u^{\varepsilon}$ is the solution of (\ref{RLogKGE1}), $u^{\varepsilon,n}_{j}$ is the numerical approximation of the (\ref{RLogKGE1}).
\begin{theorem}\label{CNFDerr}
Under the assumption (A), there exist $h_{0}>0,\tau_{0}>0$  sufficiently small and independent of $\varepsilon,$ for any $0<\varepsilon \ll1$, when $0<h\leq h_{0}$ and $0<\tau\leq \tau_{0}$, the CNFD (\ref{CNFD}) with (\ref{initialvalue1}) and (\ref{initialvalue}) satisfies the following error estimates
\begin{align}\label{CNFDerror}
\|\delta_{x}^{+}e^{\varepsilon,n}\|_{l^{2}} + \|e^{\varepsilon,n}\|_{l^{2}}\lesssim e^{\frac{T}{2\varepsilon}}\left( h^{2}+\frac{\tau^{2}}{\varepsilon^{2}}\right),~~\|u^{\varepsilon,n}\|_{l^{\infty}}\leq \Lambda+1.
\end{align}
\end{theorem}
\begin{theorem}\label{SIFD1err}
Assume  $\tau\lesssim h$ and under the assumption (A), there exist $h_{0}>0,\tau_{0}>0$  sufficiently small and independent of $\varepsilon,$  for any $0<\varepsilon \ll1$, when $0<h\leq h_{0}$ and $0<\tau\leq \tau_{0}$ and under the stability condition (\ref{sifd1stability}), the SIEFD (\ref{SIFD1}) with (\ref{initialvalue1}) and (\ref{initialvalue}) satisfies the following error estimates
\begin{align}\label{SIFD1error}
\|\delta_{x}^{+}e^{\varepsilon,n}\|_{l^{2}} + \|e^{\varepsilon,n}\|_{l^{2}}\lesssim e^{\frac{T}{2\varepsilon}}\left( h^{2}+\frac{\tau^{2}}{\varepsilon^{2}}\right),~~\|u^{\varepsilon,n}\|_{l^{\infty}}\leq {\Lambda}+1.
\end{align}
\end{theorem}
\begin{remark}
 \cite{bao2012analysis,bao2019long}
Extending to $2$ and $3$ dimensions, the above Theorems are still valid under the conditions $0<h\lesssim \sqrt{C_{d}(h)},~0<\tau\lesssim \sqrt{C_{d}(h)}$. Besides, the inverse inequality becomes
\begin{align}
\|u^{\varepsilon,n}\|_{l^{\infty}}\lesssim \frac{1}{C_{d}(h)}\left(\|\delta_{x}^{+}u^{\varepsilon,n}\|_{l^{2}} + \|u^{\varepsilon,n}\|_{l^{2}} \right),
\end{align}
where $C_{d}(h)=1/|\ln h|$ when $d=2$ and when $d=3$, $C_{d}(h)=h^{1/2}$.
\end{remark}
\subsection{Proof of Theorem \ref{CNFDerr} for the CNFD}
Define the local truncation error for the CNFD (\ref{CNFD}) as
\begin{align}\label{trunerre}
\begin{split}
\xi^{\varepsilon,0}_{j}:=&\delta_{t}^{+}u^{\varepsilon}(x_{j},0)-\gamma(x_{j})-\frac{\tau}{2}\left[ \delta^{2}_{x}\phi(x_{j})-\phi(x_{j})- \phi(x_{j})\ln(\varepsilon^{2}+(\phi(x_{j}))^{2})\right],\\
\xi^{\varepsilon,n}_{j}:=&\delta_{t}^{2}u^{\varepsilon}(x_{j},t_{n})-\frac{1}{2}\delta_{x}^{2} \big( u^{\varepsilon}(x_{j},t_{n+1})+u^{\varepsilon}(x_{j},t_{n-1}) \big)+\frac{1}{2} \big( u^{\varepsilon}(x_{j},t_{n+1})+u^{\varepsilon}(x_{j},t_{n-1}) \big)\\
&+G_{\varepsilon}\big(u^{\varepsilon}(x_{j},t_{n+1}),u^{\varepsilon}(x_{j},t_{n-1})\big),~~j\in \mathcal{T}_{N},~~n\geq1,
\end{split}
\end{align}
then we have the following bounds for the local truncation error.
\begin{lemma}\label{trunerrlem}
Under the assumption (A), we have
\begin{align}
&\|\xi^{\varepsilon,0}\|_{H^{1}}\lesssim h^{2}+\tau^{2},\\
&\|\xi^{\varepsilon, n}\|_{l^{2}}\lesssim h^{2}+\frac{\tau^{2}}{\varepsilon^{2}},\\
&\|\delta^{+}_{x}\xi^{\varepsilon, n}\|_{l^{2}}\lesssim h^{2}+\frac{\tau^{2}}{\varepsilon^{3}},~~1\leq n \leq \frac{T}{\tau}-1.
\end{align}
\end{lemma}
\begin{proof}
By (\ref{initialvalue}), it leads to
\begin{align}
|\xi^{\varepsilon,0}_{j}|\leq\frac{\tau^{2}}{6}\|\partial_{t}^{3}u^{\varepsilon}\|_{L^{\infty}}+\frac{\tau h}{6}\|\partial_{x}^{3}\phi\|_{L^{\infty}}\lesssim h^{2}+\tau^{2},
\end{align}
where the $L^{\infty}$-norm means $\|u^{\varepsilon}\|_{L^{\infty}}:=\sup\limits_{0\leq t\leq T}\sup\limits_{x\in \Omega}|u^{\varepsilon}(x,t)|$.
Similarly, we have
\begin{align}
|\delta_{x}^{+}\xi^{\varepsilon,0}_{j}|\leq\frac{\tau^{2}}{6}\|\partial_{tttx}u^{\varepsilon}\|_{L^{\infty}}+\frac{\tau h}{6}\|\partial_{x}^{4}\phi\|_{L^{\infty}}\lesssim h^{2}+\tau^{2},~~j\in \mathcal{T}_{N}.
\end{align}
Therefore,
\begin{align}
\|\xi^{\varepsilon,0}\|_{H^{1}}\lesssim h^{2}+\tau^{2}.
\end{align}
For $n\geq1$ and $j\in \mathcal{T}_{N}$, according to the Taylor expansion for the nonlinear part $G_{\varepsilon}$ at $(u^{\varepsilon}(x_{j},t_{n}))^{2}$ and noticing (\ref{G}), we denote
\begin{align}
\Gamma_{j}^{n} &:=\frac{1}{\tau}\left(\left|u^{\varepsilon}\left(x_{j}, t_{n+1}\right)\right|^{2}-\left|u^{\varepsilon}\left(x_{j}, t_{n}\right)\right|^{2}\right)=\int_{0}^{1} \partial_{t}\left(\left|u^{\varepsilon}\right|^{2}\right)\left(x_{j}, t_{n}+s \tau\right) \mathrm{d} s, \\
\widetilde{\Gamma}_{j}^{n} &:=\frac{2}{\tau^{2}}\left(\frac{1}{2}\left(\left|u^{\varepsilon}\left(x_{j}, t_{n+1}\right)\right|^{2}+\left|u^{\varepsilon}\left(x_{j}, t_{n-1}\right)\right|^{2}\right)-\left|u^{\varepsilon}\left(x_{j}, t_{n}\right)\right|^{2}\right) \\
&=\int_{0}^{1} \int_{-\theta}^{\theta} \partial_{t t}\left(\left|u^{\varepsilon}\right|^{2}\right)\left(x_{j}, t_{n}+s \tau\right) \mathrm{d} s \mathrm{d} \theta.\notag
\end{align}
Noticing that
\begin{align}
\begin{split}
\xi^{\varepsilon,n}_{j}:=&\delta_{t}^{2}u^{\varepsilon}(x_{j},t_{n})-\frac{1}{2}\delta_{x}^{2} \big( u^{\varepsilon}(x_{j},t_{n+1})+u^{\varepsilon}(x_{j},t_{n-1}) \big)+\frac{1}{2} \big( u^{\varepsilon}(x_{j},t_{n+1})+u^{\varepsilon}(x_{j},t_{n-1}) \big)\\
&+G_{\varepsilon}\big(u^{\varepsilon}(x_{j},t_{n+1}),u^{\varepsilon}(x_{j},t_{n-1})\big) \\
&-\Big[ \partial_{tt}u^{\varepsilon}(x_{j},t_{n})-\partial_{xx} u^{\varepsilon}(x_{j},t_{n})+ u^{\varepsilon}(x_{j},t_{n})+ u^{\varepsilon}(x_{j},t_{n})f_{\varepsilon}\left( ( u^{\varepsilon}(x_{j},t_{n}))^{2} \right)\Big]\\
=&\left[ \delta_{t}^{2}u^{\varepsilon}(x_{j},t_{n})-\partial_{tt}u^{\varepsilon}(x_{j},t_{n})\right] - \left[\frac{1}{2}\delta_{x}^{2} \left( u^{\varepsilon}(x_{j},t_{n+1})+u^{\varepsilon}(x_{j},t_{n-1}) \right)-\partial_{xx}u^{\varepsilon}(x_{j},t_{n}) \right]\\
&+\frac{1}{2} \big( u^{\varepsilon}(x_{j},t_{n+1})+u^{\varepsilon}(x_{j},t_{n-1}) \big) -u^{\varepsilon}(x_{j},t_{n})\\
&+\Big[ G_{\varepsilon}\big(u^{\varepsilon}(x_{j},t_{n+1}),u^{\varepsilon}(x_{j},t_{n-1})\big) - u^{\varepsilon}(x_{j},t_{n})f_{\varepsilon}\left( ( u^{\varepsilon}(x_{j},t_{n}))^{2} \right) \Big],
\end{split}
\end{align}
\begin{align}
\begin{split}
 &G_{\varepsilon}\big(u^{\varepsilon}(x_{j},t_{n+1}),u^{\varepsilon}(x_{j},t_{n-1})\big) - u^{\varepsilon}(x_{j},t_{n})f_{\varepsilon}\left( ( u^{\varepsilon}(x_{j},t_{n}))^{2} \right)\\
 =&\int_{0}^{1} f_{\varepsilon}\Big(\theta \big(u^{\varepsilon}(x_{j},t_{n+1})\big)^{2}+(1-\theta)\big(u^{\varepsilon}(x_{j},t_{n-1})\big)^{2}\Big)\mathrm{d} \theta\cdot \frac{u^{\varepsilon}(x_{j},t_{n+1})+u^{\varepsilon}(x_{j},t_{n-1})}{2}\\
&- \int_{0}^{1}f_{\varepsilon}\big( \left( u^{\varepsilon}(x_{j},t_{n})\right)^{2} \big) \mathrm{d}\theta\cdot u^{\varepsilon}(x_{j},t_{n})\\
=&\int_{0}^{1} f_{\varepsilon}\Big(\theta \big(u^{\varepsilon}(x_{j},t_{n+1})\big)^{2}+(1-\theta)\big(u^{\varepsilon}(x_{j},t_{n-1})\big)^{2}\Big)\mathrm{d} \theta\cdot \frac{u^{\varepsilon}(x_{j},t_{n+1})+u^{\varepsilon}(x_{j},t_{n-1})}{2}\\
&- \int_{0}^{1} f_{\varepsilon}\left( (u^{\varepsilon}(x_{j},t_{n}))^{2}\right)\mathrm{d} \theta\cdot \frac{u^{\varepsilon}(x_{j},t_{n+1})+u^{\varepsilon}(x_{j},t_{n-1})}{2}\\
&+ f_{\varepsilon}\left( (u^{\varepsilon}(x_{j},t_{n}))^{2}\right)\cdot \frac{u^{\varepsilon}(x_{j},t_{n+1})+u^{\varepsilon}(x_{j},t_{n-1})}{2}-f_{\varepsilon}\left( ( u^{\varepsilon}(x_{j},t_{n}))^{2} \right)\cdot u^{\varepsilon}(x_{j},t_{n}),
\end{split}
\end{align}
\begin{align}
\begin{split}
&\int_{0}^{1} f_{\varepsilon}\Big(\theta \big(u^{\varepsilon}(x_{j},t_{n+1})\big)^{2}+(1-\theta)\big(u^{\varepsilon}(x_{j},t_{n-1})\big)^{2}\Big)\mathrm{d} \theta\\
=&\int_{0}^{1}\bigg[f_{\varepsilon}\big( (u^{\varepsilon}(x_{j},t_{n}))^{2}\big) + \Big[ \theta \big(u^{\varepsilon}(x_{j},t_{n+1})\big)^{2} + (1-\theta) \big(u^{\varepsilon}(x_{j},t_{n-1})\big)^{2} - \big(u^{\varepsilon}(x_{j},t_{n})\big)^{2}\Big]f^{\prime}_{\varepsilon}\big( (u^{\varepsilon}(x_{j},t_{n}))^{2}\big)\bigg]\mathrm{d} \theta\\
&  + \int_{0}^{1}\int_{(u^{\varepsilon}(x_{j},t_{n}))^{2}}^{\theta(u^{\varepsilon}(x_{j},t_{n+1}))^{2}+(1-\theta)(u^{\varepsilon}(x_{j},t_{n-1}))^{2}} \bigg[\Big(\theta (u^{\varepsilon}(x_{j},t_{n+1}))^{2} + (1-\theta) (u^{\varepsilon}(x_{j},t_{n-1}))^{2} \Big)f^{\prime\prime}_{\varepsilon}( s) \bigg] \mathrm{d} s\mathrm{d} \theta\\
=&f_{\varepsilon}\big( (u^{\varepsilon}(x_{j},t_{n}))^{2}\big) +\frac{\tau^{2}}{2}f^{\prime}_{\varepsilon}\big(u^{\varepsilon}(x_{j},t_{n})\big)\widetilde{\Gamma}^{n}_{j}\\
&+\int_{0}^{1}\int_{(u^{\varepsilon}(x_{j},t_{n}))^{2}}^{\theta(u^{\varepsilon}(x_{j},t_{n+1}))^{2}+(1-\theta)(u^{\varepsilon}(x_{j},t_{n-1}))^{2}} \bigg[\Big(\theta (u^{\varepsilon}(x_{j},t_{n+1}))^{2} + (1-\theta) (u^{\varepsilon}(x_{j},t_{n-1}))^{2} \Big)f^{\prime\prime}_{\varepsilon}( s) \bigg] \mathrm{d} s\mathrm{d} \theta\\
=&f_{\varepsilon}\big( (u^{\varepsilon}(x_{j},t_{n}))^{2}\big) +\frac{\tau^{2}}{2}f^{\prime}_{\varepsilon}(u^{\varepsilon}(x_{j},t_{n}))\widetilde{\Gamma}^{n}_{j}\\
&+\int_{0}^{1} \int_{0}^{1} \bigg[\theta \big(u^{\varepsilon}(x_{j},t_{n+1})\big)^{2} + (1-\theta) \big(u^{\varepsilon}(x_{j},t_{n-1})\big)^{2} -\big(u^{\varepsilon}(x_{j},t_{n})\big)^{2}\bigg]\bigg[\theta \big(u^{\varepsilon}(x_{j},t_{n+1})\big)^{2} + (1-\theta) \big(u^{\varepsilon}(x_{j},t_{n-1})\big)^{2} \\
&-\delta \Big(\theta \big(u^{\varepsilon}(x_{j},t_{n+1})\big)^{2} + (1-\theta) \big(u^{\varepsilon}(x_{j},t_{n-1})\big)^{2}\Big) -(1-\delta)\big(u^{\varepsilon}(x_{j},t_{n})\big)^{2} \bigg]f^{\prime\prime}_{\varepsilon}(\xi_{j}(\theta,\delta))\mathrm{d} \delta \mathrm{d} \theta\\
=&f_{\varepsilon}\big( (u^{\varepsilon}(x_{j},t_{n}))^{2}\big) +\frac{\tau^{2}}{2}f^{\prime}_{\varepsilon}\big(u^{\varepsilon}(x_{j},t_{n})\big)\widetilde{\Gamma}^{n}_{j}\\
&+\tau^{2} \int_{0}^{1}\int_{0}^{1}(1-\delta) \big(\theta \Gamma^{n}_{j}-(1-\theta)\Gamma^{n-1}_{j}\big)^{2}f^{\prime\prime}_{\varepsilon}(\xi_{j}(\theta,\delta))\mathrm{d} \delta \mathrm{d} \theta.
\end{split}
\end{align}
So, we have
\begin{align}
\begin{split}
 &G_{\varepsilon}\big(u^{\varepsilon}(x_{j},t_{n+1}),u^{\varepsilon}(x_{j},t_{n-1})\big) - u^{\varepsilon}(x_{j},t_{n})f_{\varepsilon}\big( ( u^{\varepsilon}(x_{j},t_{n}))^{2} \big)\\
 =&\bigg[\frac{\tau^{2}}{2}f^{\prime}_{\varepsilon}\big(u^{\varepsilon}(x_{j},t_{n})\big)\widetilde{\Gamma}^{n}_{j}\\
&+\tau^{2} \int_{0}^{1}\int_{0}^{1}(1-\delta) \big(\theta \Gamma^{n}_{j}-(1-\theta)\Gamma^{n-1}_{j}\big)^{2}f^{\prime\prime}_{\varepsilon}\big(\xi_{j}(\theta,\delta)\big)\mathrm{d} \delta \mathrm{d} \theta \bigg]\cdot \frac{1}{2}\big(u^{\varepsilon}(x_{j},t_{n+1})+u^{\varepsilon}(x_{j},t_{n-1})\big)\\
&+\frac{ \tau^{2}}{2} f_{\varepsilon}\big((u^{\varepsilon}(x_{j},t_{n}))^{2}\big)\int_{-1}^{1}(1-|s|)\partial_{tt}u^{\varepsilon}(x_{j},t_{n}+s\tau)\mathrm{d}s.
\end{split}
\end{align}
Taking the Taylor expansion, we obtain
\begin{align}
\xi^{\varepsilon,n}_{j}=\frac{\tau^{2}}{12}\alpha^{\varepsilon,n}_{j}+\frac{\tau^{2}}{2}\beta^{\varepsilon,n}_{j}+\frac{h^{2}}{12}\eta^{\varepsilon,n}_{j}+\frac{\tau^{2}}{2}\phi^{\varepsilon,n}_{j}+\frac{\tau^{2}}{2}\psi^{\varepsilon,n}_{j},
\end{align}
where
\begin{align*}
\alpha^{\varepsilon,n}_{j}=&\int^{1}_{-1}(1-|s|)^{3}\partial^{4}_{t} u^{\varepsilon}(x_{j},t_{n}+s\tau)\mathrm{d}s,~~\beta^{\varepsilon,n}_{j}=\int^{1}_{-1}(1-|s|)\partial^{2}_{t}\partial^{2}_{x} u^{\varepsilon}(x_{j},t_{n}+s\tau)\mathrm{d}s,\\
\eta^{\varepsilon,n}_{j}=&\int^{1}_{-1}(1-|s|)^{3}\left( \partial^{4}_{x} u^{\varepsilon}(x_{j}+sh,t_{n+1})+ \partial^{4}_{x} u^{\varepsilon}(x_{j}+sh,t_{n-1}) \right) \mathrm{d}s,\\
\phi^{\varepsilon,n}_{j}=&\int^{1}_{-1}(1-|s|)\partial^{2}_{t} u^{\varepsilon}(x_{j},t_{n}+s\tau)\mathrm{d}s,\\
\psi^{\varepsilon,n}_{j}=&\bigg[f^{\prime}_{\varepsilon}\big(u^{\varepsilon}(x_{j},t_{n})\big)\widetilde{\Gamma}^{n}_{j}+2\int_{0}^{1}\int_{0}^{1}(1-\delta) \big(\theta \Gamma^{n}_{j}-(1-\theta)\Gamma^{n-1}_{j}\big)^{2}f^{\prime\prime}_{\varepsilon}(\xi_{j}(\theta,\delta))\mathrm{d} \delta \mathrm{d} \theta \bigg]\cdot \frac{1}{2}\big(u^{\varepsilon}(x_{j},t_{n+1})\\
&+u^{\varepsilon}(x_{j},t_{n-1})\big)+ f_{\varepsilon}\big((u^{\varepsilon}(x_{j},t_{n}))^{2}\big)\int_{-1}^{1}(1-|s|)\partial_{tt}u^{\varepsilon}(x_{j},t_{n}+s\tau)\mathrm{d}s.
\end{align*}
Under the assumption (A), by using the triangle inequality, noticing $f_{\varepsilon}\in C^{2}([0,\infty))$, and
\begin{align*}
&|f_{\varepsilon}|=|\ln(\varepsilon^{2}+(u^{\varepsilon})^{2})|\leq \ln\frac{1}{\varepsilon^{2}},\\
&|f_{\varepsilon}^{\prime}|=\left|\frac{2u^{\varepsilon}}{\varepsilon^{2}+(u^{\varepsilon})^{2}}\right|\leq\frac{1}{\varepsilon}, \\ &|f_{\varepsilon}^{\prime\prime}|=\left|\frac{2(\varepsilon^{2}-(u^{\varepsilon})^{2})}{\left(\varepsilon^{2}+(u^{\varepsilon})^{2}\right)^{2}}\right|\lesssim \frac{1}{\varepsilon^{2}}.
\end{align*}

We have
\begin{align}
\begin{split}
|\xi^{\varepsilon,n}_{j}|\lesssim &h^{2}\| \partial^{4}_{x} u^{\varepsilon}\|_{L^{\infty}}\\
&+\tau^{2}\bigg[\|\partial^{4}_{t} u^{\varepsilon}\|_{L^{\infty}}+\|\partial^{2}_{x}\partial^{2}_{t} u^{\varepsilon}\|_{L^{\infty}}+\|\partial^{2}_{t} u^{\varepsilon}\|_{L^{\infty}}+\|\partial^{2}_{t}u^{\varepsilon}\|_{L_{\infty}}\|f_{\varepsilon}((u^{\varepsilon})^{2})\|_{L^{\infty}}\\
&~~~~~~+\Big(\|\partial_{t}(u^{\varepsilon})^{2}\|_{L^{\infty}}\|f^{\prime\prime}_{\varepsilon}((u^{\varepsilon})^{2})\|_{L^{\infty}}+\|\partial_{tt}(u^{\varepsilon})^{2}\|_{L^{\infty}}\|f^{\prime}_{\varepsilon}((u^{\varepsilon})^{2})\|_{L^{\infty}}\Big)\|u^{\varepsilon}\|_{L^{\infty}}\bigg]\\
\lesssim &h^{2}+\frac{\tau^{2}}{\varepsilon^{2}}.
\end{split}
\end{align}
Noticing $f_{\varepsilon}\in C^{3}([0,\infty)),f_{\varepsilon}^{\prime\prime\prime}=\frac{-12(u^{\varepsilon})^{3}+4u^{\varepsilon}\varepsilon^{2}}{(\varepsilon^{2}+(u^{\varepsilon})^{2})^{3}} ,$ we have $|f_{\varepsilon}^{\prime\prime\prime}|\lesssim \frac{1}{\varepsilon^{3}} $. With the similar method, we have
\begin{align}
\begin{split}
|\delta_{x}^{+}\xi^{\varepsilon,n}_{j}|\lesssim &h^{2}\| \partial^{5}_{x} u^{\varepsilon}\|_{L^{\infty}}\\
&+\tau^{2}\Bigg[\|\partial^{4}_{t}\partial_{x} u^{\varepsilon}\|_{L^{\infty}}+\|\partial^{3}_{x}\partial^{2}_{t} u^{\varepsilon}\|_{L^{\infty}}+\|\partial^{2}_{t}\partial_{x} u^{\varepsilon}\|_{L^{\infty}}\\
&+\bigg(\|\partial^{2}_{t}u^{\varepsilon}\|_{L^{\infty}}\|f^{\prime}_{\varepsilon}((u^{\varepsilon})^{2})\|_{L^{\infty}}+\big(\|\partial_{t}(u^{\varepsilon})^{2}\|_{L^{\infty}}\|f^{\prime\prime\prime}_{\varepsilon}((u^{\varepsilon})^{2})\|_{L^{\infty}}\\
&~~~~+\|\partial_{tt}(u^{\varepsilon})^{2}\|_{L^{\infty}}\|f^{\prime\prime}_{\varepsilon}((u^{\varepsilon})^{2})\|_{L^{\infty}}\big)\|u^{\varepsilon}\|_{L^{\infty}}\bigg)\cdot\|\partial_{x} (u^{\varepsilon})^{2}\|_{L^{\infty}}\\
&+\|\partial^{2}_{t}\partial_{x}u^{\varepsilon}\|_{L^{\infty}}\|f_{\varepsilon}((u^{\varepsilon})^{2})\|_{L^{\infty}}\\
&+\bigg(\|\partial_{tx}(u^{\varepsilon})^{2}\|_{L^{\infty}}\|f^{\prime\prime}_{\varepsilon}((u^{\varepsilon})^{2})\|_{L^{\infty}}+\|\partial_{ttx}(u^{\varepsilon})^{2}\|_{L^{\infty}}\|f^{\prime}_{\varepsilon}((u^{\varepsilon})^{2})\|_{L^{\infty}}\bigg)\cdot\|u^{\varepsilon}\|_{L^{\infty}}\\
&+\bigg(\|\partial_{t}(u^{\varepsilon})^{2}\|_{L^{\infty}}\|f^{\prime\prime}_{\varepsilon}((u^{\varepsilon})^{2})\|_{L^{\infty}}+\|\partial_{tt}(u^{\varepsilon})^{2}\|_{L^{\infty}}\|f^{\prime}_{\varepsilon}((u^{\varepsilon})^{2})\|_{L^{\infty}}\bigg)\cdot\|\partial_{x}u^{\varepsilon}\|_{L^{\infty}}\Bigg]\\
\lesssim &h^{2}+\tau^{2}\max\left\{\ln\left(\frac{1}{\varepsilon^{2}}\right), \frac{1}{\varepsilon^{2}},\frac{1}{\varepsilon^{3}}\right\}\\
\lesssim &h^{2}+\frac{\tau^{2}}{\varepsilon^{3}}.
\end{split}
\end{align}

This ends the proof.
\end{proof}

For the CNFD (\ref{CNFD}), we establish the error estimates in Theorem \ref{CNFDerr}. The proof is different from the schemes of the EFD and the SIFD \cite{yan2020regularized} of the RLogKGE (\ref{RLogKGE}). The main difficulty of the proof are dealing with the nonlinearity and bounding the numerical solution $u^{\varepsilon,n}$, i.e., $\|u^{\varepsilon,n}\|_{l^{\infty}}\lesssim 1$. Following the idea in \cite{akrivis1991fully,bao2013optimal,bao2012uniform}, we truncate the nonlinearity $f_{\varepsilon}$ to a global Lipschitz function with compact support in $d$-dimensions ($d=1,2,3$). And the error can be obtained if the numerical solution is close to the bounded continuous solution. In this paper, we apply the same idea. Choosing a smooth function $\rho(s)\in C^{\infty}(\mathbb{R})$ such that
\begin{align*}
&\rho(s)=\left\{\begin{array}{ll}
{1,} & {0 \leq|s| \leq 1}, \\
{\in[0,1],} & {1 \leq|s| \leq 2}, \\
{0,} & {|s| \geq 2}.
\end{array}\right.
\end{align*}
Denote $B=(\Lambda+1)^{2}, ~f_{B}(s)=f_{\varepsilon}(s)\rho(s/B), ~ F_{B}(s)=\int^{s}_{0}f_{B}(\sigma)\mathrm{d}\sigma,~ \rho_{_B}(s)=\rho(s/B)$, where $s\geq0,~s\in \mathbb{R}$. Then $f_{B}(s),~F_{B}(s)$ have compact support and are smooth, global Lipschitz continous, i.e., there exists $C_{B}=\|f^{\prime}_{B}\|_{L^{\infty}}$ , such that
\begin{align}
|f_{B}(s_{1})-f_{B}(s_{2})|\leq C_{B}|\sqrt{s_{1}}-\sqrt{s_{2}}|,~~\forall s_{1},s_{2}\geq0,~s_{1},s_{2} \in \mathbb{R}.
\end{align}

Let $\hat{u}^{\varepsilon,0}=u^{\varepsilon,0},~\hat{u}^{\varepsilon,1}=u^{\varepsilon,1}$ and we determine $\hat{u}^{\varepsilon,n+1} \in X_{N},~\text{for}~ n\geq1$ by
\begin{align}\label{CNFDB}
\delta_{t}^{2}\hat{u}^{\varepsilon,n}_{j}-\frac{1}{2}\delta_{x}^{2} (\hat{u}^{\varepsilon,n+1}_{j}+\hat{u}^{\varepsilon,n-1}_{j})+\frac{1}{2} (\hat{u}^{\varepsilon,n+1}_{j}+\hat{u}^{\varepsilon,n-1}_{j})+ G_{_B}(\hat{u}^{\varepsilon,n+1}_{j},\hat{u}^{\varepsilon,n-1}_{j})=0,~~ j\in \mathcal{T}_{N};
\end{align}
where $G_{_B}(z_{1},z_{2})$ for $z_{1},z_{2}\in \mathbb{R}$ is
\begin{align}
\begin{aligned}\label{G_{_B}}
G_{_B}\left(z_{1}, z_{2}\right) &=\int_{0}^{1} f_{B}\left( \theta z_{1}^{2}+(1-\theta)z_{2}^{2}\right) \mathrm{d} \theta \cdot g_{_B}\left(\frac{z_{1}+z_{2}}{2}\right) \\
&=\frac{F_{B}\left(z_{1}^{2}\right)-F_{B}\left(z_{2}^{2}\right)}{z_{1}^{2}-z_{2}^{2}} \cdot g_{_B}\left(\frac{z_{1}+z_{2}}{2}\right),
\end{aligned}
\end{align}
and $\hat{u}^{\varepsilon,n}_{j}$ can be viewed as another approximation of $u^{\varepsilon}(x,t)$. According to Lemma \ref{CNFDsol}, (\ref{CNFDB}) is uniquely solvable for small $\tau$. Denote the error $\chi^{\varepsilon,n}$ for $n\geq1$ as
\begin{align}\label{errorhat}
\chi^{\varepsilon,n}_{j}=u^{\varepsilon}(x_{j},t_{n})-\hat{u}^{\varepsilon,n}_{j}.
\end{align}
We can get the following estimates:
\begin{theorem}\label{CNFDBerr}
Under the assumption (A), there exist $h_{0}>0,\tau_{0}>0$  sufficiently small and independent of $\varepsilon,$ for any $0<\varepsilon \ll1$, when $0<h\leq h_{0}$ and $0<\tau\leq \tau_{0}$, the CNFD (\ref{CNFDB}) with (\ref{initialvalue1}) and (\ref{initialvalue}) satisfies the following error estimates
\begin{align}\label{CNFDBerror}
\|\delta_{x}^{+}\chi^{\varepsilon,n}\|_{l^{2}} + \|\chi^{\varepsilon,n}\|_{l^{2}}\lesssim e^{\frac{T}{2\varepsilon}}\left(h^{2}+\frac{\tau^{2}}{\varepsilon^{2}}\right),~~\|\hat{u}^{\varepsilon,n}\|_{l^{\infty}}\leq \Lambda+1.
\end{align}
\end{theorem}
Define the local truncation error $\hat{\xi}^{\varepsilon,n}_{j}\in X_{N}$ of (\ref{CNFDB}) for $j\in \mathcal{T}_{N},~n\geq1$ as
\begin{align}\label{trunerreB}
\begin{split}
\hat{\xi}^{\varepsilon,0}_{j}:=&\delta_{t}^{+}u^{\varepsilon}(x_{j},0)-\gamma(x_{j})-\frac{\tau}{2}\big[ \delta^{2}_{x}\phi(x_{j})-\phi(x_{j})- \phi(x_{j})\ln(\varepsilon^{2}+(\phi(x_{j}))^{2})\big],\\
\hat{\xi}^{\varepsilon,n}_{j}:=&\delta_{t}^{2}u^{\varepsilon}(x_{j},t_{n})-\frac{1}{2}\delta_{x}^{2} \big( u^{\varepsilon}(x_{j},t_{n+1})+u^{\varepsilon}(x_{j},t_{n-1}) \big)+\frac{1}{2} \big( u^{\varepsilon}(x_{j},t_{n+1})+u^{\varepsilon}(x_{j},t_{n-1}) \big)\\
&+G_{_B}\big(u^{\varepsilon}(x_{j},t_{n+1}),u^{\varepsilon}(x_{j},t_{n-1})\big).
\end{split}
\end{align}
Similar to Lemma \ref{trunerrlem}, we have the following local truncation error estimates and we omit the proof.
\begin{lemma}\label{trunerrBlem}
Under assumption (A), we have
\begin{align}
&\|\hat{\xi}^{\varepsilon,0}\|_{H^{1}}\lesssim h^{2}+\tau^{2},\\
&\|\hat{\xi}^{\varepsilon, n}\|_{l^{2}}\lesssim h^{2}+\frac{\tau^{2}}{\varepsilon^{2}},\\
&\|\delta^{+}_{x}\hat{\xi}^{\varepsilon, n}\|_{l^{2}}\lesssim h^{2}+\frac{\tau^{2}}{\varepsilon^{3}},~~1\leq n \leq \frac{T}{\tau}-1.
\end{align}
\end{lemma}
Next, we give the error bounds of the nonlinear term as follows.
\begin{lemma}\label{nonerrBlem}
For $j\in \mathcal{T}_{N}^{0}$ and $1\leq n \leq \frac{T}{\tau}-1$, we define the error of the nonlinear term
\begin{align}
\hat{\eta}^{\varepsilon, n}_{j}=G_{_B}\big(u^{\varepsilon}(x_{j},t_{n+1}),u^{\varepsilon}(x_{j},t_{n-1})\big)-G_{_B}(\hat{u}^{\varepsilon,n+1}_{j},\hat{u}^{\varepsilon,n-1}_{j}),
\end{align}
we have
\begin{align}\label{nonlinearty}
&\|\hat{\eta}^{\varepsilon, n}\|_{l^{2}}\lesssim \frac{1}{\varepsilon}\left(\|\chi^{\varepsilon, n+1}\|_{l^{2}}+\|\chi^{\varepsilon, n-1}\|_{l^{2}}\right),\\
&\|\delta^{+}_{x}\hat{\eta}^{\varepsilon, n}\|_{l^{2}}\lesssim \frac{1}{\varepsilon^{2}}\left(\|\chi^{\varepsilon, n+1}\|_{l^{2}}+\|\chi^{\varepsilon, n-1}\|_{l^{2}}+\|\delta^{+}_{x}\chi^{\varepsilon, n+1}\|_{l^{2}}+\|\delta^{+}_{x}\chi^{\varepsilon, n-1}\|_{l^{2}}\right).
\end{align}
\end{lemma}
\begin{proof}
We define the error of the nonlinear term
\begin{align}
\hat{\eta}^{\varepsilon, n}_{j}=G_{_B}(u^{\varepsilon}(x_{j},t_{n+1}),u^{\varepsilon}(x_{j},t_{n-1}))-G_{_B}(\hat{u}^{\varepsilon,n+1}_{j},\hat{u}^{\varepsilon,n-1}_{j}).
\end{align}
And denote
\begin{align*}
&\rho_{j}^{\varepsilon, n}(\theta) =\theta\left(u^{\varepsilon}\left(x_{j}, t_{n+1}\right)\right)^{2}+(1-\theta)\left(u^{\varepsilon}\left(x_{j}, t_{n-1}\right)\right)^{2} ,\\
&\hat{\rho}_{j}^{\varepsilon, n}(\theta) =\theta\left(\hat{u}_{j}^{\varepsilon, n+1}\right)^{2}+(1-\theta)\left(\ \hat{u}_{j}^{\varepsilon, n-1}\right)^{2}, \\
&\mu_{j}^{\varepsilon, n} =\frac{1}{2}\big[u^{\varepsilon}\left(x_{j}, t_{n+1}\right)+u^{\varepsilon}\left(x_{j}, t_{n-1}\right)\big], \\ &\hat{\mu}_{j}^{\varepsilon, n}=\frac{1}{2}\left[\hat{u}_{j}^{\varepsilon, n+1}+\hat{u}_{j}^{\varepsilon, n-1}\right], \\
&\pi_{j}^{\varepsilon, n} =u^{\varepsilon}\left(x_{j}, t_{n}\right)+\hat{u}_{j}^{\varepsilon, n},~~~~j \in \mathcal{T}_{N}^{0},~1\leq n \leq \frac{T}{\tau}-1,~\theta\in [0,1].
\end{align*}
From the definition of $G_{B},~F_{B},~g_{_B}$, we can get
\begin{align*}
\hat{\eta}^{\varepsilon, n}_{j}=& g_{_B}(\mu_{j}^{\varepsilon, n}) \int_{0}^{1}\left[f_{B}\left(\rho_{j}^{\varepsilon, n}(\theta)\right)-f_{B}\left(\hat{\rho}_{j}^{\varepsilon, n}(\theta)\right)\right] \mathrm{d} \theta \\
&+\left[g_{_B}\left(\mu_{j}^{\varepsilon, n}\right)-g_{_B}\left(\hat{\mu}_{j}^{\varepsilon, n}\right)\right] \int_{0}^{1} f_{B}\left(\hat{\rho}_{j}^{\varepsilon, n}(\theta)\right) \mathrm{d} \theta.
\end{align*}
According to the Lipschitz property of $f_{B}(s^{2})$, we obtain
\begin{align}\label{f_{B}}
\begin{split}
\left|f_{B}\left(\rho_{j}^{\varepsilon, n}(\theta)\right)-f_{B}\left(\hat{\rho}_{j}^{\varepsilon, n}(\theta)\right)\right|&\leq \|f_{B}^{\prime}\|_{L^{\infty}}
\left|\sqrt{\rho_{j}^{\varepsilon, n}(\theta)}-\sqrt{\hat{\rho}_{j}^{\varepsilon, n}(\theta)}\right| \\
&\leq \frac{1}{\varepsilon}\frac{\theta \pi_{j}^{\varepsilon, n+1}\left|\chi_{j}^{\varepsilon, n+1}\right|+(1-\theta) \pi_{j}^{\varepsilon, n-1}\left|\chi_{j}^{\varepsilon, n-1}\right|}{\sqrt{\rho_{j}^{\varepsilon, n}(\theta)}+\sqrt{\hat{\rho}_{j}^{\varepsilon, n}(\theta)}} \\
& \leq \frac{1}{\varepsilon}\left(\sqrt{\theta}\left|\chi_{j}^{\varepsilon, n+1}\right|+\sqrt{1-\theta}\left|\chi_{j}^{\varepsilon, n-1}\right|\right).
\end{split}
\end{align}
By the Lipschitz property of $g_{_B}(s^{2})$, it follows that
\begin{align}\label{nonlinearty1}
\begin{split}
\left|g_{_B}\left(\mu_{j}^{\varepsilon, n}\right)-g_{_B}\left(\hat{\mu}_{j}^{\varepsilon, n}\right) \right|&\leq C\left|\sqrt{\mu_{j}^{\varepsilon, n}}-\sqrt{\hat{\mu}_{j}^{\varepsilon, n}}\right|\\
&\leq C\left(\left|\chi_{j}^{\varepsilon, n+1}\right|+\left|\chi_{j}^{\varepsilon, n-1}\right|\right).
\end{split}
\end{align}
Then with the boundness of $\mu_{j}^{\varepsilon, n},f_{B}\left(\hat{\rho}_{j}^{\varepsilon, n}(\theta)\right)$, and according to (\ref{f_{B}}), (\ref{nonlinearty1}), we have
\begin{align}
\left|\hat{\eta}^{\varepsilon, n}_{j} \right|\lesssim \frac{1}{\varepsilon}\left(\left|\chi_{j}^{\varepsilon, n+1}\right|+\left|\chi_{j}^{\varepsilon, n-1}\right|\right).
\end{align}
For $\delta^{+}_{x}\hat{\eta}^{\varepsilon, n}_{j}$, it is easy to get
\begin{align*}
\delta^{+}_{x}\hat{\eta}^{\varepsilon, n}_{j}=& \delta^{+}_{x}{\bigg{[}}g_{_B}(\hat{\mu}_{j}^{\varepsilon, n}) \int_{0}^{1}\left[f_{B}\left(\rho_{j}^{\varepsilon, n}(\theta)\right)-f_{B}\left(\hat{\rho}_{j}^{\varepsilon, n}(\theta)\right)\right] \mathrm{d }\theta \\
&+\left[g_{_B}\left(\mu_{j}^{\varepsilon, n}\right)-g_{_B}\left(\hat{\mu}_{j}^{\varepsilon, n}\right)\right] \int_{0}^{1} f_{B}\left(\rho_{j}^{\varepsilon, n}(\theta)\right) \mathrm{d} \theta {\bigg{]}}\\
=&g_{_B}(\hat{\mu}_{j}^{\varepsilon, n} )\int_{0}^{1}\delta^{+}_{x}\left[f_{B}\left(\rho_{j}^{\varepsilon, n}(\theta)\right)-f_{B}\left(\hat{\rho}_{j}^{\varepsilon, n}(\theta)\right)\right] \mathrm{d} \theta \\
&+\left[g_{_B}\left(\mu_{j}^{\varepsilon, n}\right)-g_{_B}\left(\hat{\mu}_{j}^{\varepsilon, n}\right)\right] \int_{0}^{1} \delta^{+}_{x}f_{B}\left(\rho_{j}^{\varepsilon, n}(\theta)\right) \mathrm{d} \theta\\
&+\delta^{+}_{x}g_{_B}(\hat{\mu}_{j}^{\varepsilon, n} )\int_{0}^{1}\left[f_{B}\left(\rho_{j+1}^{\varepsilon, n}(\theta)\right)-f_{B}\left(\hat{\rho}_{j+1}^{\varepsilon, n}(\theta)\right)\right] \mathrm{d} \theta \\
&+\delta^{+}_{x}\left[g_{_B}\left(\mu_{j}^{\varepsilon, n}\right)-g_{_B}\left(\hat{\mu}_{j}^{\varepsilon, n}\right)\right] \int_{0}^{1} f_{B}\left(\rho_{j+1}^{\varepsilon, n}(\theta)\right) \mathrm{d} \theta.
\end{align*}
Firstly, for $j\in \mathcal{T}_{N},~\theta,s \in[0,1]$, we define
\begin{align*}
&\kappa_{j}^{\varepsilon, n}(\theta,s)=s \rho_{j+1}^{\varepsilon, n}(\theta)+(1-s) \rho_{j}^{\varepsilon, n}(\theta),\\ &\hat{\kappa}_{j}^{\varepsilon, n}(\theta,s)=s \hat{\rho}_{j+1}^{\varepsilon, n}(\theta)+(1-s) \hat{\rho}_{j}^{\varepsilon, n}(\theta),
\end{align*}
then we have
\begin{align*}
&\delta^{+}_{x}\left[f_{B}\left(\rho_{j}^{\varepsilon, n}(\theta)\right)-f_{B}\left(\hat{\rho}_{j}^{\varepsilon, n}(\theta)\right) \right]\\
=&\delta_{x}^{+} \rho_{j}^{\varepsilon, n}(\theta) \int_{0}^{1} f_{B}^{\prime}\left(\kappa_{j}^{\varepsilon, n}(\theta, s)\right) \mathrm{d } s-\delta_{x}^{+} \hat{\rho}_{j}^{\varepsilon, n}(\theta) \int_{0}^{1} f_{B}^{\prime}\left(\hat{\kappa}_{j}^{\varepsilon, n}(\theta, s)\right) \mathrm{d} s\\
=&\int_{0}^{1}\left[f_{B}^{\prime}\left(\kappa_{j}^{\varepsilon, n}(\theta, s)\right)-f_{B}^{\prime}\left(\hat{\kappa}_{j}^{\varepsilon, n}(\theta, s)\right)\right] \delta_{x} \rho_{j}^{\varepsilon, n}(\theta) \mathrm{d} s\\
&+\int_{0}^{1} f_{B}^{\prime}\left(\hat{\kappa}_{j}^{\varepsilon, n}(\theta, s)\right)\left[\delta_{x}^{+} \rho_{j}^{\varepsilon, n}(\theta)-\delta_{x}^{+} \hat{\rho}_{j}^{\varepsilon, n}(\theta)\right] \mathrm{d }s,
\end{align*}
and
\begin{align*}
&\delta^{+}_{x}\left[ \rho_{j}^{\varepsilon, n}(\theta)-\hat{\rho}_{j}^{\varepsilon, n}(\theta)\right]\\
=&\theta \bigg[ 2u^{\varepsilon}\left(x_{j}, t_{n+1}\right) \delta_{x}^{+}\chi_{j}^{\varepsilon, n+1}+2\chi_{j+1}^{\varepsilon, n+1}\delta_{x}^{+} u^{\varepsilon}\left(x_{j}, t_{n+1}\right)\\
&~~~~~-\chi_{j}^{\varepsilon, n+1}\delta_{x}^{+} \chi_{j}^{\varepsilon,n+1}-\chi_{j+1}^{\varepsilon, n+1}\delta_{x}^{+} \chi_{j}^{\varepsilon,n+1}\bigg]\\
&+(1-\theta)\bigg[ 2u^{\varepsilon}\left(x_{j}, t_{n-1}\right) \delta_{x}^{+}\chi_{j}^{\varepsilon, n-1}+2\chi_{j+1}^{\varepsilon, n-1}\delta_{x}^{+} u^{\varepsilon}\left(x_{j}, t_{n-1}\right)\\
&~~~~~~~~~~~~~~~-\chi_{j}^{\varepsilon, n-1}\delta_{x}^{+} \chi_{j}^{\varepsilon,n-1}-\chi_{j+1}^{\varepsilon, n-1}\delta_{x}^{+} \chi_{j}^{\varepsilon,n-1}\bigg].
\end{align*}
Besides,
\begin{align*}
&\sqrt{1-\theta}\left|\chi_{j+1}^{\varepsilon, n-1}\right| \leq \sqrt{\hat{\rho}_{j+1}^{\varepsilon, n}(\theta)}+\left|u^{\varepsilon}\left(x_{j+1}, t_{n-1}\right)\right|,\\ &\sqrt{1-\theta}\left|\chi_{j}^{\varepsilon, n-1}\right| \leq \sqrt{\hat{\rho}_{j}^{\varepsilon, n}(\theta)}+\left|u^{\varepsilon}\left(x_{j}, t_{n-1}\right)\right|,\\
&\sqrt{\theta}\left|\chi_{j+1}^{\varepsilon, n+1}\right| \leq \sqrt{\hat{\rho}_{j+1}^{\varepsilon, n}(\theta)}+\left|u^{\varepsilon}\left(x_{j+1}, t_{n+1}\right)\right|,\\
&\sqrt{\theta}\left|\chi_{j}^{\varepsilon, n+1}\right| \leq \sqrt{\hat{\rho}_{j}^{\varepsilon, n}(\theta)}+\left|u^{\varepsilon}\left(x_{j}, t_{n+1}\right)\right|.
\end{align*}
By the Lipschitz property of $f_{B}$, and $f_{B}^{\prime}$, we have
\begin{align*}
&\left|\int_{0}^{1}\left[f_{B}^{\prime}\left(\kappa_{j}^{\varepsilon, n}(\theta, s)\right)-f_{B}^{\prime}\left(\hat{\kappa}_{j}^{\varepsilon, n}(\theta, s)\right)\right] \mathrm{d} s\right|\\
\leq&\|f_{B}^{\prime\prime} \|_{L^{\infty}}\int_{0}^{1}\left|\sqrt{\kappa_{j}^{\varepsilon, n}(\theta, s)}-\sqrt{\hat{\kappa}_{j}^{\varepsilon, n}(\theta, s)}\right| \mathrm{d} s\\
\lesssim& \frac{1}{\varepsilon^{2}}\left(\left|\chi_{j+1}^{\varepsilon, n+1}\right|+\left|\chi_{j}^{\varepsilon, n+1}\right|+ \left|\chi_{j+1}^{\varepsilon, n-1}\right|+\left|\chi_{j}^{\varepsilon, n-1}\right|\right),
\end{align*}
\begin{align*}
&\left|\int_{0}^{1} f_{B}^{\prime}\left(\hat{\kappa}_{j}^{\varepsilon, n}(\theta, s)\right) \mathrm{d} s\right|=\left|\frac{f_{B}\left(\hat{\rho}_{j+1}^{\varepsilon, n}(\theta)\right)-f_{B}\left(\hat{\rho}_{j}^{\varepsilon, n}(\theta)\right)}{\hat{\rho}_{j+1}^{\varepsilon, n}(\theta)-\hat{\rho}_{j}^{\varepsilon, n}(\theta)}\right| \leq \frac{1}{\varepsilon\left(\sqrt{\hat{\rho}_{j+1}^{\varepsilon, n}(\theta)}+\sqrt{\hat{\rho}_{j}^{\varepsilon, n}(\theta)}\right)}.
\end{align*}
Then according to the boundedness of $\delta_{x}^{+} \rho_{j}^{\varepsilon, n}(\theta),~g_{_B}(\cdot) \text { and } f_{B}^{\prime}(\cdot)$, we arrive at
\begin{align}\label{v_{1}}
\begin{split}
&\left|\int_{0}^{1} \delta_{x}^{+}\left[f_{B}\left(\rho_{j}^{\varepsilon, n}(\theta)\right)-f_{B}\left(\hat{\rho}_{j}^{\varepsilon, n}(\theta)\right)\right] \mathrm{d} \theta \cdot g_{_B}\left(\hat{\mu}_{j}^{\varepsilon, n}\right)\right|\\
\lesssim & \sum_{m=n+1, n-1}\frac{1}{\varepsilon^{2}}\left(\left|\chi_{j}^{\varepsilon, m}\right|+\left|\chi_{j+1}^{\varepsilon, m}\right|+\left|\delta_{x}^{+} \chi_{j}^{\varepsilon, m}\right|\right).
\end{split}
\end{align}
Secondly, in view of the boundedness of $\delta_{x}^{+} f_{B}\left(\rho_{j}^{\varepsilon, n}(\theta)\right)$ as well as the Lipschitz property of $g_{_B}(z)$, we get
\begin{align}\label{v_{2}}
\left|\int_{0}^{1} \delta_{x}^{+} f_{B}\left(\rho_{j}^{\varepsilon, n}(\theta)\right) \mathrm{d} \theta \cdot\left[g_{_B}\left(\mu_{j}^{\varepsilon, n}\right)-g_{_B}\left(\hat{\mu}_{j}^{\varepsilon, n}\right)\right]\right| \lesssim \frac{1}{\varepsilon}\left(\left|\chi_{j}^{\varepsilon, n+1}\right|+\left|\chi_{j}^{\varepsilon, n-1}\right|\right).
\end{align}
Thirdly, noticing the property of $g_{_B}(\cdot) \in C_{0}^{\infty}$, we have
\begin{align*}
|\delta^{+}_{x}g_{_B}(\hat{\mu}_{j}^{\varepsilon, n})|\lesssim |\delta^{+}_{x}\hat{\mu}_{j}^{\varepsilon, n}|\leq |\delta^{+}_{x}\chi_{j}^{\varepsilon, n+1}|+|\delta^{+}_{x}\chi_{j}^{\varepsilon, n-1}|+C.
\end{align*}
Recalling the property of $f_{B}(s)$ and (\ref{f_{B}}), we obtian
\begin{align}\label{v_{3}}
\begin{split}
&\left|\int_{0}^{1}\left[f_{B}\left(\rho_{j+1}^{\varepsilon, n}(\theta)\right)-f_{B}\left(\hat{\rho}_{j+1}^{\varepsilon, n}(\theta)\right)\right] \mathrm{d} \theta \cdot \delta_{x}^{+} g_{_B}\left(\hat{\mu}_{j}^{\varepsilon, n}\right)\right|\\
\lesssim &\sum_{m=n-1, n+1}\frac{1}{\varepsilon}\left(\left|\chi_{j+1}^{\varepsilon, m}\right|+\left|\delta_{x}^{+} \chi_{j}^{\varepsilon, m}\right|\right).
\end{split}
\end{align}
Finally, we define
\begin{align*}
\sigma_{j}^{n}(\theta)=\theta \mu_{j+1}^{\varepsilon, n}+(1-\theta) \mu_{j}^{\varepsilon, n}, \quad \hat{\sigma}_{j}^{n}(\theta)=\theta \hat{\mu}_{j+1}^{\varepsilon, n}+(1-\theta) \hat{\mu}_{j}^{\varepsilon, n},
\end{align*}
for $\theta \in[0,1] \text { and } j\in \mathcal{T}_{N}$.
Then we have
\begin{align}\label{v_{4}}
\begin{split}
&\left|\delta_{x}^{+}\left(g_{_B}\left(\mu_{j}^{\varepsilon, n}\right)-g_{_B}\left(\hat{\mu}_{j}^{\varepsilon, n}\right)\right)\right|\\
=&\left|\delta_{x}^{+}\left[\rho_{_B}\left(\left(\mu_{j}^{\varepsilon, n}\right)^{2}\right) \mu_{j}^{\varepsilon, n}-\rho_{_B}\left(\left(\hat{\mu}_{j}^{\varepsilon, n}\right)^{2}\right) \hat{\mu}_{j}^{\varepsilon, n}\right]\right|\\
=&\left|\int_{0}^{1}\left[\delta_{x}^{+} \mu_{j}^{\varepsilon, n} \partial_{z} g_{_B}(z)-\delta_{x}^{+} \hat{\mu}_{j}^{\varepsilon, n} \partial_{\hat{z}} g_{_B}(\hat{z})\right] \mathrm{d} \theta\right|\\
\leq&\left|\int_{0}^{1}\left(\partial_{z} g_{B}\left(\sigma_{j}^{n}(\theta)\right)-\partial_{z} g_{B}\left(\hat{\sigma}_{j}^{n}(\theta)\right)\right) \delta_{x}^{+} \mu_{j}^{\varepsilon, n} \mathrm{d} \theta\right| \\
&+\left|\int_{0}^{1}\left(\delta_{x}^{+} \mu_{j}^{\varepsilon, n}-\delta_{x}^{+} \hat{\mu}_{j}^{\varepsilon, n}\right) \partial_{z} g_{B}\left(\hat{\sigma}_{j}^{n}(\theta)\right) \mathrm{d} \theta\right| \\
\lesssim &\max _{\theta \in[0,1]}\left\{| \sigma_{j}^{n}(\theta)- \hat{\sigma}_{j}^{n}(\theta)|\right\}+\left|\delta_{x}^{+}\left(\chi_{j}^{\varepsilon, n+1}+\chi_{j}^{\varepsilon, n-1}\right)\right| \\
\lesssim & \sum_{m=n+1, n-1}\left(\left|\chi_{j}^{\varepsilon, m}\right|+\left|\chi_{j+1}^{\varepsilon, m}\right|+\left|\delta_{x}^{+} \chi_{j}^{\varepsilon, m}\right|\right) ,
\end{split}
\end{align}
where $\partial_{z} g_{_B}(z)=\rho_{_B}\left(z^{2}\right)+2z^{2} \rho_{_B}^{\prime}\left(z^{2}\right),~z=\sigma_{j}^{n}(\theta),~\hat{z}=\hat{\sigma}_{j}^{n}(\theta)$.\\
Combining (\ref{v_{1}})-(\ref{v_{4}}) and the H\"{o}lder inequality, this ends the proof.
\end{proof}

According to Lemma \ref{trunerrBlem} and Lemma \ref{nonerrBlem}, we give the proof of Theorem \ref{CNFDBerr}. Subtracting (\ref{CNFDB}) from (\ref{trunerreB}), we have
\begin{subequations}\label{energyB}
\begin{align}
&\delta_{t}^{2}\chi^{\varepsilon,n}_{j}-\frac{1}{2}\delta_{x}^{2} \left( \chi^{\varepsilon,n+1}_{j}+\chi^{\varepsilon,n-1}_{j} \right)+\frac{1}{2} \left( \chi^{\varepsilon,n+1}_{j}+\chi^{\varepsilon,n-1}_{j} \right)=\hat{\xi}^{\varepsilon,n}_{j}-\hat{\eta}^{\varepsilon,n}_{j},~~n\geq1,\label{energyB1}\\
&\chi^{\varepsilon,0}=0,~~\chi^{\varepsilon,1}_{j}=\tau\hat{\xi}^{\varepsilon,0}_{j},~~j\in \mathcal{T}_{N}^{0}.\label{energyB2}
\end{align}
\end{subequations}
Denote the `energy' for the error vector $\chi^{\varepsilon,n}$ as
\begin{align}
\hat{E}^{n}=\|\delta_{t}^{+}\chi^{\varepsilon,n}\|_{l^{2}}^{2}+\frac{1}{2}\left( \|\delta_{x}^{+}\chi^{\varepsilon,n}\|_{l^{2}}^{2}+\|\delta_{x}^{+}\chi^{\varepsilon,n+1}\|_{l^{2}}^{2}\right)+\frac{1}{2}\left( \|\chi^{\varepsilon,n}\|_{l^{2}}^{2}+\|\chi^{\varepsilon,n+1}\|_{l^{2}}^{2}\right),~~n\geq 0.
\end{align}
\begin{proof}({\bf {Proof of Theorem \ref{CNFDBerr}}})
When $n=1$,
under the assumption (A), by Lemma \ref{trunerrBlem} we can conclude the errors of the first step discretization (\ref{initialvalue})
\begin{align}\label{localtrunerror}
\chi^{\varepsilon,0}=0,~~\|\chi^{\varepsilon,1}_{j}\|_{H^{1}}\lesssim h^{2}+\tau^{2},
\end{align}
for sufficiently small $0<\tau<\tau_{1}$ and $0<h<h_{1}$. So it is true for $n=0, 1$. Next, we prove the Theorem \ref{CNFDBerr} for $2\leq n \leq \frac{T}{\tau}-1$.

Multiplying both sides of $(\ref{energyB})$ by $h(\chi^{\varepsilon,n+1}_{j}-\chi^{\varepsilon,n-1}_{j})$, summing up for $j$, applying Lemma $\ref{trunerrBlem}$ and Lemma $\ref{nonerrBlem}$ and making use of the Young's inequality, we have
\begin{align}
\begin{split}
\hat{E}^{n}-\hat{E}^{n-1}&=h\sum^{N-1}_{j=0}(\hat{\xi}^{\varepsilon,n}_{j}-\hat{\eta}^{\varepsilon,n}_{j})(\chi^{\varepsilon,n+1}_{j}-\chi^{\varepsilon,n-1}_{j})\\
&\leq h\sum^{N-1}_{j=0}\left(\left|\hat{\xi}^{\varepsilon,n}_{j}\right|+\left|\hat{\eta}^{\varepsilon,n}_{j}\right|\right)\left|\chi^{\varepsilon,n+1}_{j}-\chi^{\varepsilon,n-1}_{j}\right|\\
&= \tau h\sum^{N-1}_{j=0}\left(\left|\hat{\xi}^{\varepsilon,n}_{j}\right|+\left|\hat{\eta}^{\varepsilon,n}_{j}\right|\right)\left|\delta^{+}_{t}\chi^{\varepsilon,n+1}_{j}+\delta^{+}_{t}\chi^{\varepsilon,n-1}_{j}\right|\\
&\leq \tau \left[\left(\left\|\hat{\xi}^{\varepsilon,n}\right\|^{2}_{l^{2}}+\left\|\hat{\eta}^{\varepsilon,n}\right\|^{2}_{l^{2}}\right)+\left\|\delta^{+}_{t}\chi^{\varepsilon,n}\right\|^{2}_{l^{2}}+\left\|\delta^{+}_{t}\chi^{\varepsilon,n-1}\right\|^{2}_{l^{2}}\right]\\
&\lesssim \tau \bigg[\left(h^{2}+\frac{\tau^{2}}{\varepsilon^{2}}\right)^{2}+\frac{1}{\varepsilon}\left(\| \chi^{\varepsilon,n+1}_{j}\|^{2}_{l^{2}}+ \|\chi^{\varepsilon,n-1}_{j}\|^{2}_{l^{2}}\right)\\
&~~~+\|\delta^{+}_{t}\chi^{\varepsilon,n}_{j}\|^{2}_{l^{2}}+ \|\delta^{+}_{t}\chi^{\varepsilon,n-1}_{j}\|^{2}_{l^{2}}\bigg]\\
&\lesssim \tau \left[\left(h^{2}+\frac{\tau^{2}}{\varepsilon^{2}}\right)^{2}+\frac{1}{\varepsilon} \left(\hat{E}^{n}+\hat{E}^{n-1} \right)\right],~n\geq1.
\end{split}
\end{align}
Therefore, there exists a constant $\tau_{2}<\varepsilon$ sufficiently small and independent of $h$, such that when $0<\tau\leq \tau_{2}$, we get
\begin{align}
\hat{E}^{n}-\hat{E}^{n-1}\lesssim \tau \left[\frac{1}{\varepsilon} \hat{E}^{n-1}+\left(h^{2}+\frac{\tau^{2}}{\varepsilon^{2}}\right)^{2}\right],~n\geq1.
\end{align}
Summing the above inequality for time steps from $1$ to $n$, when $\tau<\frac{1}{2}$, it follows that
\begin{align}
\hat{E}^{n}-\hat{E}^{0}
\lesssim\frac{\tau}{\varepsilon}\sum_{m=0}^{n-1}\hat{E}^{m}+\left( h^{2}+\frac{\tau^{2}}{\varepsilon^{2}}\right)^{2},~1\leq n\leq \frac{T}{\tau}-1.
\end{align}
Besides, noticing
\begin{align}
\hat{E}^{0}=\|\delta_{t}^{+}\chi^{\varepsilon,0}\|_{l^{2}}^{2}+\frac{1}{2}\left( \|\delta_{x}^{+}\chi^{\varepsilon,0}\|_{l^{2}}^{2}+\|\delta_{x}^{+}\chi^{\varepsilon,1}\|_{l^{2}}^{2}\right)+\frac{1}{2}\left( \|\chi^{\varepsilon,0}\|_{l^{2}}^{2}+\|\chi^{\varepsilon,1}\|_{l^{2}}^{2}\right).
\end{align}
According to (\ref{localtrunerror}), (\ref{energyB1}), (\ref{energyB2}), we get
\begin{align}
\hat{E}^{0}\lesssim (h^{2}+\tau^{2})^{2}.
\end{align}
With the Gronwall's inequality  \cite{holte2009discrete}, there exists a constant $\tau_{3}$ sufficiently small and independent of $h$, such that when $0<\tau\leq \tau_{3}$, we obtain
\begin{align}
\hat{E}^{n}\lesssim e^{ \frac{T}{\varepsilon}}\left(h^{2}+\frac{\tau^{2}}{\varepsilon^{2}}\right)^{2},~1\leq n\leq \frac{T}{\tau}-1.
\end{align}
Noticing $\|\chi^{\varepsilon,n}\|^{2}_{l^{2}}+\|\delta^{+}_{x}\chi^{\varepsilon,n}\|^{2}_{l^{2}}\leq 2\hat{E}^{n}$, we can obtain
\begin{align}
\|\chi^{\varepsilon,n}\|_{l^{2}}+\|\delta^{+}_{x}\chi^{\varepsilon,n}\|_{l^{2}}\lesssim e^{\frac{T}{2\varepsilon}}\left(h^{2}+\frac{\tau^{2}}{\varepsilon^{2}}\right),~1\leq n\leq \frac{T}{\tau}-1.
\end{align}
Besides, by the discrete Sobolev inequality, the following holds
\begin{align}
\begin{split}
\|\chi^{\varepsilon,n}\|_{l^{\infty}}\leq &\|\chi^{\varepsilon,n}\|_{l^{2}}+\|\delta^{+}_{x}\chi^{\varepsilon,n}\|_{l^{2}}\\
\lesssim &e^{\frac{T}{2\varepsilon}}\left(h^{2}+\frac{\tau^{2}}{\varepsilon^{2}}\right),~1\leq n\leq \frac{T}{\tau}-1.
\end{split}
\end{align}
Therefore, there exists a constant $\tau_{4}>0, h_{2}>0$ sufficiently small when $0<\tau\leq \tau_{4},0<h\leq h_{2}$, we obtain
\begin{align}
\|\hat{u}^{n}\|_{l^{\infty}}\leq \|u(x,t_{n})\|_{L^{\infty}}+\|\chi^{\varepsilon,n}\|_{l^{\infty}}\leq \Lambda+1.
\end{align}
We complete the proof by choosing $h_{0}=\min\{h_{1},h_{2}\},\tau_{0}=\min\{\tau_{1},\tau_{2},\tau_{3},\tau_{4}\}$.
\end{proof}

\begin{proof}({\bf {Proof of Theorem \ref{CNFDerr}}})
Recalling the definition of $\rho$, Theorem \ref{CNFDBerr} implies that (\ref{CNFDB}) collapses to (\ref{CNFD}). According to the unique solvability of the CNFD scheme (\ref{CNFD}), $\hat{u}^{\varepsilon,n}$ is identical to $u^{\varepsilon,n}$. Therefore Theorem \ref{CNFDerr} is a direct consequence of Theorem \ref{CNFDBerr}.
\end{proof}
\subsection{Proof of Theorem \ref{SIFD1err} for the SIEFD}
Determine another approximation to $u(x_{j},t_{n})$ of the scheme SIEFD (\ref{SIFD1}) from
\begin{align}\label{SIFD1B}
\delta_{t}^{2}\hat{u}^{\varepsilon,n}_{j}-\delta_{x}^{2} \hat{u}^{\varepsilon,n}_{j}+\frac{1}{2} (\hat{u}^{\varepsilon,n+1}_{j}+\hat{u}^{\varepsilon,n-1}_{j})+ G_{_B}(\hat{u}^{\varepsilon,n+1}_{j},\hat{u}^{\varepsilon,n-1}_{j})=0,~~ j\in \mathcal{T}_{N};
\end{align}
\begin{theorem}\label{SIFD1Berr}
Assume  $\tau\lesssim h$ and under the assumption (A), there exist $h_{0}>0,\tau_{0}>0$  sufficiently small and independent of $\varepsilon,$  for any $0<\varepsilon \ll1$, when $0<h\leq h_{0}$ and $0<\tau\leq \tau_{0}$ and under the stability condition (\ref{sifd1stability}), the SIEFD (\ref{SIFD1B}) with (\ref{initialvalue1}) and (\ref{initialvalue}) satisfies the following error estimates
\begin{align}\label{SIFD1Berror}
\|\delta_{x}^{+}\chi^{\varepsilon,n}\|_{l^{2}} + \|\chi^{\varepsilon,n}\|_{l^{2}}\lesssim e^{\frac{T}{2\varepsilon}}\left(h^{2}+\frac{\tau^{2}}{\varepsilon^{2}}\right),~~\|\hat{u}^{\varepsilon,n}\|_{l^{\infty}}\leq \Lambda+1,
\end{align}
where $\chi^{\varepsilon,n}_{j}=u^{\varepsilon}(x_{j},t_{n})-\hat{u}^{\varepsilon,n}_{j}$.
\end{theorem}
We define the local truncation error $\hat{\xi}^{\varepsilon,n}_{j}\in X_{N}$ for $j\in \mathcal{T}_{N},~n\geq1$ as
\begin{align}\label{SIFD1trunerreB}
\begin{split}
\hat{\xi}^{\varepsilon,0}_{j}:=&\delta_{t}^{+}u^{\varepsilon}(x_{j},0)-\gamma(x_{j})-\frac{\tau}{2}\left[ \delta^{2}_{x}\phi(x_{j})-\phi(x_{j})- \phi(x_{j})\ln(\varepsilon^{2}+(\phi(x_{j}))^{2})\right],\\
\hat{\xi}^{\varepsilon,n}_{j}:=&\delta_{t}^{2}u^{\varepsilon}(x_{j},t_{n})-\delta_{x}^{2} u^{\varepsilon}(x_{j},t_{n})+\frac{1}{2} \left( u^{\varepsilon}(x_{j},t_{n+1})+u^{\varepsilon}(x_{j},t_{n-1}) \right)\\
&+G_{_B}(u^{\varepsilon}(x_{j},t_{n+1}),u^{\varepsilon}(x_{j},t_{n-1})).
\end{split}
\end{align}
\begin{lemma}\label{SIFD1trunerrBlem}
Under assumption (A), we have
\begin{align}
&\|\hat{\xi}^{\varepsilon,0}\|_{H^{1}}\lesssim h^{2}+\tau^{2},\\
&\|\hat{\xi}^{\varepsilon, n}\|_{l^{2}}\lesssim h^{2}+\frac{\tau^{2}}{\varepsilon^{2}},\\
&\|\delta^{+}_{x}\hat{\xi}^{\varepsilon, n}\|_{l^{2}}\lesssim h^{2}+\frac{\tau^{2}}{\varepsilon^{3}},~~1\leq n \leq \frac{T}{\tau}-1.
\end{align}
\end{lemma}
\begin{proof}
The proof is similar to Lemma \ref{trunerrlem}, therefore we omit it.
\end{proof}

Next, we give the proof of Theorem \ref{SIFD1Berr}.

\begin{proof}({\bf {Proof of Theorem \ref{SIFD1Berr}}})
The proof is similar to Theorem \ref{CNFDBerr}. With the results of Lemma \ref{SIFD1trunerrBlem} and \ref{nonerrBlem}, we can immediately prove Theorem \ref{SIFD1Berr}.
\end{proof}

\begin{proof}({\bf {Proof of Theorem \ref{SIFD1err}}})
Recalling the definition of $\rho$, Theorem \ref{SIFD1Berr} implies that (\ref{SIFD1B}) collapses to (\ref{SIFD1}). According to the unique solvability of the SIEFD scheme (\ref{SIFD1}), $\hat{u}^{\varepsilon,n}$ is identical to $u^{\varepsilon,n}$. Therefore Theorem \ref{SIFD1err} is a direct consequence of Theorem \ref{SIFD1Berr}.
\end{proof}
\section{Numerical results}
In this section, we first quantify the error estimates of the regularized model and the CNFD scheme (\ref{CNFD}). Then we investigate the well simulation of the LogKGE (\ref{LogKGE}). Since the results of the SIEFD (\ref{SIFD1}) are similar to those of the CNFD (\ref{CNFD}), we omit the details for brevity.
\subsection{Accuracy test}
Here we take $d=1, \lambda=1$ and define the error functions as:
\begin{align}
&\hat{e}^{\varepsilon}(t_{n}):=u(\cdot,t_{n})-u^{\varepsilon}(\cdot,t_{n}),
 ~~~~~~e^{\varepsilon}(t_{n}):=u^{\varepsilon}(\cdot,t_{n})-u^{\varepsilon,n}, \\
 &\tilde{e}^{\varepsilon}(t_{n}):=u(\cdot,t_{n})-u^{\varepsilon,n}.
\end{align}
Besides, we denote the error functions:
\begin{align}
&e^{\varepsilon}_{\infty}(t_{n}):=\|u^{\varepsilon}(\cdot,t_{n})-u^{\varepsilon,n}\|_{l^{\infty}}, \quad  e^{\varepsilon}_{2}(t_{n}):=\|u^{\varepsilon}(\cdot,t_{n})-u^{\varepsilon,n}\|_{l^{2}}, \\
&e^{\varepsilon}_{H^{1}}(t_{n}):=\sqrt{(e^{\varepsilon}_{2}(t_{n}))^{2}+\|\delta^{+}_{x}(u^{\varepsilon}(\cdot,t_{n})-u^{\varepsilon,n})\|^{2}_{l^{2}}}.
\end{align}
Here $u, u^{\varepsilon}$ are the exact solutions of the LogKGE (\ref{LogKGE}) and the RLogKGE (\ref{RLogKGE}), $u^{n}, u^{\varepsilon,n}$ are the numerical solutions of the LogKGE (\ref{LogKGE}) and the RLogKGE (\ref{RLogKGE}).

$\mathbf{Example~1}$.
The initial data is set as $\phi(x)=e^{-\frac{k^{2}x^{2}}{2(c^{2}-k^{2})}},\gamma(x)=\frac{ckx}{c^2-k^2}e^{-\frac{(kx)^{2}}{2(c^{2}-k^{2})}}$, and the Gaussian solitary wave solution is
\begin{align}
u(x,t)=e^{-\frac{(kx-ct)^{2}}{2(c^{2}-k^{2})}},
\end{align}\label{Gaussion}
where $c = 2, k=1$. The RLogKGE (\ref{RLogKGE}) is simulated on the interval $\Omega=[-16,16]$ with periodic boundary conditions. The `exact' solution $u^{\varepsilon}$ is obtained numerically by the CNFD (\ref{CNFD}) scheme with very fine time step and mesh size $\tau=0.01\times2^{-9}, h=2^{-10}$.

\subsubsection{Convergence of the regularized model}

Here we test the convergence rate between the solutions of the RLogKGE (\ref{RLogKGE}) and the LogKGE (\ref{LogKGE}).  Figure \ref{fig:regexhat} depicts $\|\hat{e}^{\varepsilon}\|_{l^{2}}, \|\hat{e}^{\varepsilon}\|_{l^{\infty}},\|\hat{e}^{\varepsilon}\|_{H^{1}}$ at $t=0.5$ with different values of $\varepsilon$.

From Figure \ref{fig:regexhat}, we can conclude that the solutions of the RLogKGE (\ref{RLogKGE}) converge linearly to the LogKGE (\ref{LogKGE}) with regard to $\varepsilon$, and the convergence rate is $O(\varepsilon)$ in the $l^{2}$-norm, $l^{\infty}$-norm, $H^{1}$-norm.
\begin{center}
\begin{figure}[htbp]
\centering
\subfigure{\includegraphics[width=.43\textwidth, height=0.38\textwidth]{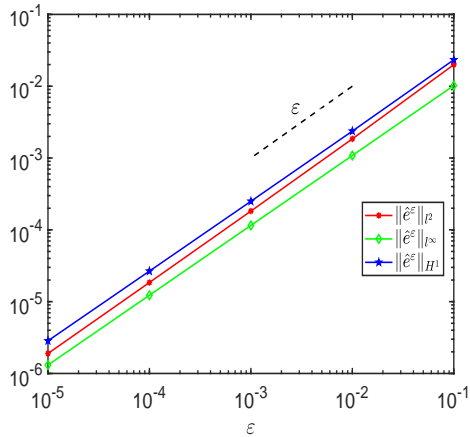}}
\caption{The errors $\hat{e}^{\varepsilon}(0.5)$ in three different norms with the scheme  CNFD $(\ref{CNFD})$.}
\label{fig:regexhat}
\end{figure}
\end{center}
\subsubsection{Convergence of FDTD to the RLogKGE}

Then we measure the error $e^{\varepsilon}$ of CNFD (\ref{CNFD})
 to the RLogKGE (\ref{RLogKGE})  for various mesh size $h$, time step $\tau$ under any fixed parameter $0<\varepsilon\ll1$ .

Firstly, we are concerned about the temporal errors of the CNFD (\ref{CNFD}) in the $l^{2}$-norm, $l^{\infty}$-norm, $H^{1}$-norm at $t=1$. The results are displayed in Table \ref{CNFDtoRlogtime} with fixed mesh size $h=2^{-10}$, and varying time step $\tau_{j}=0.1\times2^{-j}$ for $j=1,\ldots,5$ with three different values of $\varepsilon=0.1/2^{1},0.1/2^{3},0.1/2^{15}$.

Secondly, we investigate the spatial accuracy of the CNFD (\ref{CNFD}) at $t=1$, we set time step $\tau=0.01\times{2^{-9}}$, such that the errors from the time discretization are ignored then we solve the RLogKGE (\ref{RLogKGE}) with the CNFD versus mesh size $h=0.5\times2^{-j},j=1,\ldots,5$. The results are displayed in Table \ref{CNFDtoRlogspace} which tabulates $\|e^{\varepsilon}\|_{l^{2}}, \|e^{\varepsilon}\|_{l^{\infty}},\|e^{\varepsilon}\|_{H^{1}}$ with different $h$ for Example~1.

From Table \ref{CNFDtoRlogtime} and Table \ref{CNFDtoRlogspace}, we can make the observations:

(i) The scheme CNFD (\ref{CNFD}) are uniformly second order accurate which are almost independent of $\varepsilon$ in both temporal and spatial discretizations for the RLogKGE (\ref{RLogKGE}). 

 (ii) The numerical results agree and confirm the analytical results in the Theorem \ref{CNFDerr}.

\begin{table}
\centering
\caption{Temporal errors of the CNFD (\ref{CNFD}) scheme to the RLogKGE (\ref{RLogKGE}) at $t=1$}
\begin{tabular}{cccccccccccc}
\toprule  
$\|e^{\varepsilon} (1)\|_{l^{2}} $ &$\tau=0.1$ & $\tau/2 $& $\tau/2^{2}$ & $\tau/2^{3} $& $\tau/2^{4}$ &$\tau/2^{5}$&\\
\midrule  
$\varepsilon=0.1/2^{1} $&1.15E-03 &2.94E-04& 7.43E-05 &1.87E-05 & 4.70E-06&1.19E-07\\
 rate   &--& 1.96 & 1.98 & $1.99$ &1.99&1.98\\\hline
$\varepsilon=0.1/2^{3} $&1.15E-03 &2.96E-04& 7.49E-05 &1.88E-05 & 4.73E-06&1.20E-07\\
 rate   &--&  1.96 & 1.98 & $1.99$ &1.99&1.99\\\hline
 $\varepsilon=0.1/2^{15} $&3.22E-03 &8.20E-04& 2.07E-04 &5.20E-05 & 1.31E-05&3.31E-06\\
 rate   &--&  1.97 & 1.99 & $1.99$ &1.99&1.98\\
\toprule  
$\|e^{\varepsilon}(1) \|_{l^{\infty}} $  &$\tau=0.1$ & $\tau/2 $& $\tau/2^{2}$ & $\tau/2^{3} $& $\tau/2^{4}$ &$\tau/2^{5}$&\\
\midrule  
$\varepsilon=0.1/2^{1} $&7.90E-04 &2.04E-04& 5.17E-05 &1.30E-05 & 3.28E-06&8.35E-07\\
 rate   &--& 1.96 & 1.98 & $1.99$ &1.99&1.98\\\hline
$\varepsilon=0.1/2^{3} $&7.90E-04 &2.04E-04& 5.17E-05 &1.30E-05 & 3.28E-06&8.34E-07\\
 rate   &--& 1.96 & 1.98 & $1.99$ &1.99&1.98\\\hline
 $\varepsilon=0.1/2^{15} $&1.94E-03 &4.91E-04& 1.23E-04 &3.10E-05 & 7.80E-06&1.99E-06\\
 rate   &--& 1.96 & 1.98 & $1.99$ &1.99&1.97\\
\toprule  
$\|e^{\varepsilon}(1) \|_{H^{1}} $ &$\tau=0.1$ & $\tau/2 $& $\tau/2^{2}$ & $\tau/2^{3} $& $\tau/2^{4}$ &$\tau/2^{5}$&\\
\midrule  
$\varepsilon=0.1/2^{1} $&1.15E-03 &2.94E-04& 7.43E-05 &1.87E-05 & 4.71E-06 &1.23E-06\\
 rate   &--& $1.96$ & 1.98 & $1.99$ &1.99&1.93\\\hline
$\varepsilon=0.1/2^{3} $&1.15E-03 &2.96E-04& 7.48E-05 &1.88E-05 & 4.75E-06 &1.24E-06\\
 rate   &--& $1.96$ & 1.98 & $1.99$ &1.99&1.93\\\hline
 $\varepsilon=0.1/2^{15} $&3.22E-03 &8.20E-04& 2.07E-04 &5.20E-05 & 1.31E-05 &3.41E-06\\
 rate   &--& $1.97$ & 1.99 & $1.99$ &1.99&1.94\\
\bottomrule 
\end{tabular}
\label{CNFDtoRlogtime}
\end{table}
\begin{table}
\centering
\caption{Spatial errors of the CNFD (\ref{CNFD}) scheme to the RLogKGE (\ref{RLogKGE}) at $t=1$}
\begin{tabular}{cccccccccccc}
\toprule  
$\|e^{\varepsilon} (1)\|_{l^{2}} $ &$h=0.5$ & $h/2 $& $h/2^{2}$ & $h/2^{3} $& $h/2^{4}$ &$h/2^{5}$&\\
\midrule  
$\varepsilon=0.1/2^{1} $&3.43E-03 &8.61E-04& 2.16E-04 &5.39E-05 & 1.35E-05&3.38E-06\\
 rate   &--& 1.99 & 2.00 & $2.00$ &2.00&2.00\\\hline
$\varepsilon=0.1/2^{3} $&3.46E-03&8.70E-04 &2.18E-04& 5.44E-05 &1.36E-05 & 3.41E-06\\
 rate   &--&  1.99 & 2.00 & $2.00$ &2.00&2.00\\\hline
 $\varepsilon=0.1/2^{15} $&3.49E-03&8.78E-04 &2.20E-04& 5.49E-05 &1.37E-05 & 3.44E-06\\
 rate   &--&  1.99 & 2.00 & $2.00$ &2.00&2.00\\
\toprule  
$\|e^{\varepsilon}(1) \|_{l^{\infty}} $  &$h=0.5$ & $h/2 $& $h/2^{2}$ & $h/2^{3} $& $h/2^{4}$ &$h/2^{5}$&\\
\midrule  
$\varepsilon=0.1/2^{1} $&1.87E-03 &4.90E-04&1.23E-04& 3.06E-05 &7.65E-06 & 1.97E-06\\
 rate   &--& 1.94 & 2.00 & $2.01$ &2.00&1.98\\\hline
$\varepsilon=0.1/2^{3} $&1.88E-03 &4.90E-04&1.23E-04& 3.06E-05 &7.65E-06 & 1.97E-06\\
 rate   &--& 1.94 & 2.00 & $2.01$ &2.00&1.96\\\hline
 $\varepsilon=0.1/2^{15} $&1.88E-03 &4.90E-04& 1.23E-04 &3.06E-05 & 7.65E-06&1.97E-06\\
 rate   &--& 1.94 & 2.00 & $2.01$ &2.00&1.96\\
\toprule  
$\|e^{\varepsilon}(1) \|_{H^{1}} $ &$h=0.5$ & $h/2 $& $h/2^{2}$ & $h/2^{3} $& $h/2^{4}$ &$h/2^{5}$&\\
\midrule  
$\varepsilon=0.1/2^{1} $&5.10E-03 &1.29E-03& 3.24E-04 &8.09E-05 & 2.02E-05 &5.06E-06\\
 rate   &--& $1.98$ & 2.00 & $2.00$ &2.00&2.00\\\hline
$\varepsilon=0.1/2^{3} $&5.15E-03 &1.30E-03& 3.27E-04 &8.17E-05 & 2.04E-05 &5.12E-06\\
 rate   &--& $1.98$ & 2.00 & $2.00$ &2.00&2.00\\\hline
 $\varepsilon=0.1/2^{15} $&5.17E-03 &1.31E-03& 3.28E-04 &8.20E-05 & 2.05E-05 &5.14E-06\\
 rate   &--& $1.98$ & 2.00 & $2.00$ &2.00&2.00\\
\bottomrule 
\end{tabular}
\label{CNFDtoRlogspace}
\end{table}
\subsubsection{Convergence of the FDTD to the LogKGE}

Here we report the convergence rates of the CNFD (\ref{CNFD}) to the LogKGE (\ref{LogKGE}) for Example~1. Table \ref{CNFDHe} displays $l^{2}$-norm, $l^{\infty}$-norm, $H^{1}$-norm of $\tilde{e}^{\varepsilon}(1)$ for various mesh size $h$, time step $\tau$ and parameter $\varepsilon$, respectively. From Table \ref{CNFDHe}, we can draw the following conclusions:

(i): The CNFD (\ref{CNFD}) converges to the LogKGE (\ref{LogKGE}) at the rate $O(h^{2}+\tau^{2})$ only when $\varepsilon\lesssim \tau^{2}$ and $\varepsilon\lesssim h^{2}$ (cf. lower triangles below the diagonal which in bold letter in Table \ref{CNFDHe});

(ii): When $\tau^{2}\lesssim \varepsilon$ and $h^{2}\lesssim \varepsilon$ (cf. the right most column of the Table \ref{CNFDHe}), the RLogKGE (\ref{RLogKGE}) converges linearly to the LogKGE (\ref{LogKGE}) at $O(\varepsilon)$.

\begin{table}
\centering
\caption{The convergence of the CNFD (\ref{CNFD}) scheme to the LogKGE (\ref{LogKGE}) with different $\tau,h,\varepsilon$ at $t=1$}
\begin{tabular}{cccccccccccc}
\toprule  
$\|\tilde{e}^{\varepsilon} (1)\|_{l^{2}} $ &$  h=0.1 $ & $ h/2   $& $h/2^{2}   $ & $h/2^{3}    $& $h/2^{4}   $ &$h/2^{5}   $& \\
$ $ &$\tau=0.1$ & $\tau/2 $& $\tau/2^{2}$ & $\tau/2^{3} $& $\tau/2^{4}$ &$\tau/2^{5}$&\\
\midrule  
$\varepsilon=10^{-3} $&$\mathbf{1.23E}$-$\mathbf{02}$ &3.18E-03& 1.32E-03 &1.22E-03 & 1.24E-03&1.25E-03\\
 rate   &--& 1.95 & 1.27 & $0.11$ &-0.03&-0.01\\\hline
$\varepsilon/4$ &1.25E-02 &$\mathbf{3.22E}$-$\mathbf{03}$ &8.51E-04&3.85E-04 & 3.53E-04 &3.55E-04\\
 rate   &--& $\mathbf{1.95}$ & $1.92$ & $1.14$ &0.13&-0.01\\\hline
$\varepsilon/4^{2}$&1.25E-02 &3.24E-03& $\mathbf{8.26E}$-$\mathbf{04}$ &2.25E-04 & 1.11E-04 &1.02E-04\\
 rate   &--& $1.95$ &$\mathbf{ 1.97}$ & $1.88$ &1.02&0.12\\\hline
$\varepsilon/4^{3}$ &1.25E-02 &3.25E-03& 8.27E-04 &$\mathbf{2.10E}$-$\mathbf{04}$ & 5.90E-05 &3.15E-05\\
 rate   &--& $1.94$ & 1.97 & $\mathbf{1.98}$ &1.83&0.90\\\hline
$\varepsilon/4^{4}$ &1.25E-02 &3.25E-03& 8.27E-04 &2.09E-04 & $\mathbf{5.30E}$-$\mathbf{05}$ &1.54E-05\\
 rate   &--& $1.94$ & 1.97 & $1.99$ &$\mathbf{1.98}$&1.78\\
\toprule  
$\|\tilde{e}^{\varepsilon}(1) \|_{l^{\infty}} $ &$  h=0.1 $ & $ h/2   $& $h/2^{2}   $ & $h/2^{3}    $& $h/2^{4}   $ &$h/2^{5}   $& \\
$ $ &$\tau=0.1$ & $\tau/2 $& $\tau/2^{2}$ & $\tau/2^{3} $& $\tau/2^{4}$ &$\tau/2^{5}$&\\
\midrule  
$\varepsilon=10^{-3} $&7.61E-03 &$\mathbf{1.96E}$-$\mathbf{03}$& 6.92E-04 &7.54E-04 & 7.71E-04 &7.75E-04\\
 rate   &--& $\mathbf{1.96}$ & 1.50 & -0.12 &-0.03&-0.01\\\hline
$\varepsilon/4$ &7.61E-03 &1.96E-03& $\mathbf{4.96E}$-$\mathbf{04}$ &2.19E-04 & 2.27E-04 &2.29E-04\\
 rate   &--& $1.96$ & $\mathbf{1.98}$ & 1.18 &-0.05&-0.01\\\hline
$\varepsilon/4^{2}$&7.61E-03 &1.96E-03& 4.96E-04 &$\mathbf{1.25E}$-$\mathbf{04}$ & 6.64E-05 &6.73E-05\\
 rate   &--& $1.96$ & 1.98 & $\mathbf{1.99}$ &0.91&-0.02\\\hline
$\varepsilon/4^{3}$ &7.61E-03 &1.96E-03& 4.96E-04 &1.25E-04 & $\mathbf{3.12E}$-$\mathbf{04}$ &1.98E-05\\
 rate   &--& $1.96$ & 1.98 & $1.99$ &$\mathbf{2.00}$&0.66\\\hline
$\varepsilon/4^{4}$ &7.61E-03 &1.96E-03& 4.96E-04 &1.25E-04 & 3.12E-05 &$\mathbf{7.82E}$-$\mathbf{06}$\\
 rate   &--& $1.96$ & 1.98 & $1.99$ &2.00&$\mathbf{2.00}$\\
\toprule  
$\|\tilde{e}^{\varepsilon}(1) \|_{H^{1}} $ &$  h=0.1 $ & $ h/2   $& $h/2^{2}   $ & $h/2^{3}    $& $h/2^{4}   $ &$h/2^{5}   $& \\
$ $ &$\tau=0.1$ & $\tau/2 $& $\tau/2^{2}$ & $\tau/2^{3} $& $\tau/2^{4}$ &$\tau/2^{5}$&\\
\midrule  
$\varepsilon=10^{-3} $&$\mathbf{1.79E}$-$\mathbf{02}$ &4.73E-03& 1.87E-03 &1.59E-03 & 1.59E-03 &1.59E-03\\
 rate   &--& $1.91$ & 1.34 & $0.24$ &0.00&0.00\\\hline
$\varepsilon/4$ &1.79E-02  &$\mathbf{4.64E}$-$\mathbf{03}$& 1.25E-03 &5.43E-04 & 4.68E-04 &4.64E-04\\
 rate   &--& $\mathbf{1.95}$ & $1.89$ & $1.20$ &0.22&0.01\\\hline
$\varepsilon/4^{2}$&1.79E-02 &4.63E-03& $\mathbf{1.18E}$-$\mathbf{03}$ &3.26E-04& 1.55E-04 &1.37E-04\\
 rate   &--& $1.95$ & $\mathbf{1.97}$ & $1.86$ &1.08&0.18\\\hline
$\varepsilon/4^{3}$ &1.79E-02 &4.62E-03& 1.18E-04 &$\mathbf{2.99E}$-$\mathbf{04}$ & 8.46E-05 &4.39E-05\\
 rate   &--& $1.95$ & 1.98 & $\mathbf{1.97}$ &$1.82$&0.95\\\hline
$\varepsilon/4^{4}$ &1.79E-02 &4.62E-03& 1.17E-03 &2.96E-04 & $\mathbf{7.53E}$-$\mathbf{05}$ &2.20E-05\\
 rate   &--& $1.95$ & 1.98 & $1.99$ &$\mathbf{1.98}$&1.78\\
\bottomrule 
\end{tabular}
\label{CNFDHe}
\end{table}

\subsection{The simulation of numerical solution}
In this section, we apply the CNFD sheme (\ref{CNFD}) to quantify the simulation of the LogKGE (\ref{LogKGE}).

{\bf Example 2}: Taking the initial solution as
\begin{align}
\phi(x)=\cos(\pi x),~\gamma(x)=\sin(\pi x),~x\in[-1,1],
\end{align}
and the grid as $\tau=0.01, h=2^{-6}$. We picture the waveforms and the energy error obtained by the scheme CNFD (\ref{CNFD}) during $t\in [0,10]$ in Figs. \ref{fig:case2_energyerror}, \ref{fig:case2_u}. From the two pictures, we can see that the solution is stable and the energy is well conserved.

\begin{center}
\begin{figure}[htbp]
\centering
\subfigure{\includegraphics[width=.33\textwidth, height=0.30\textwidth]{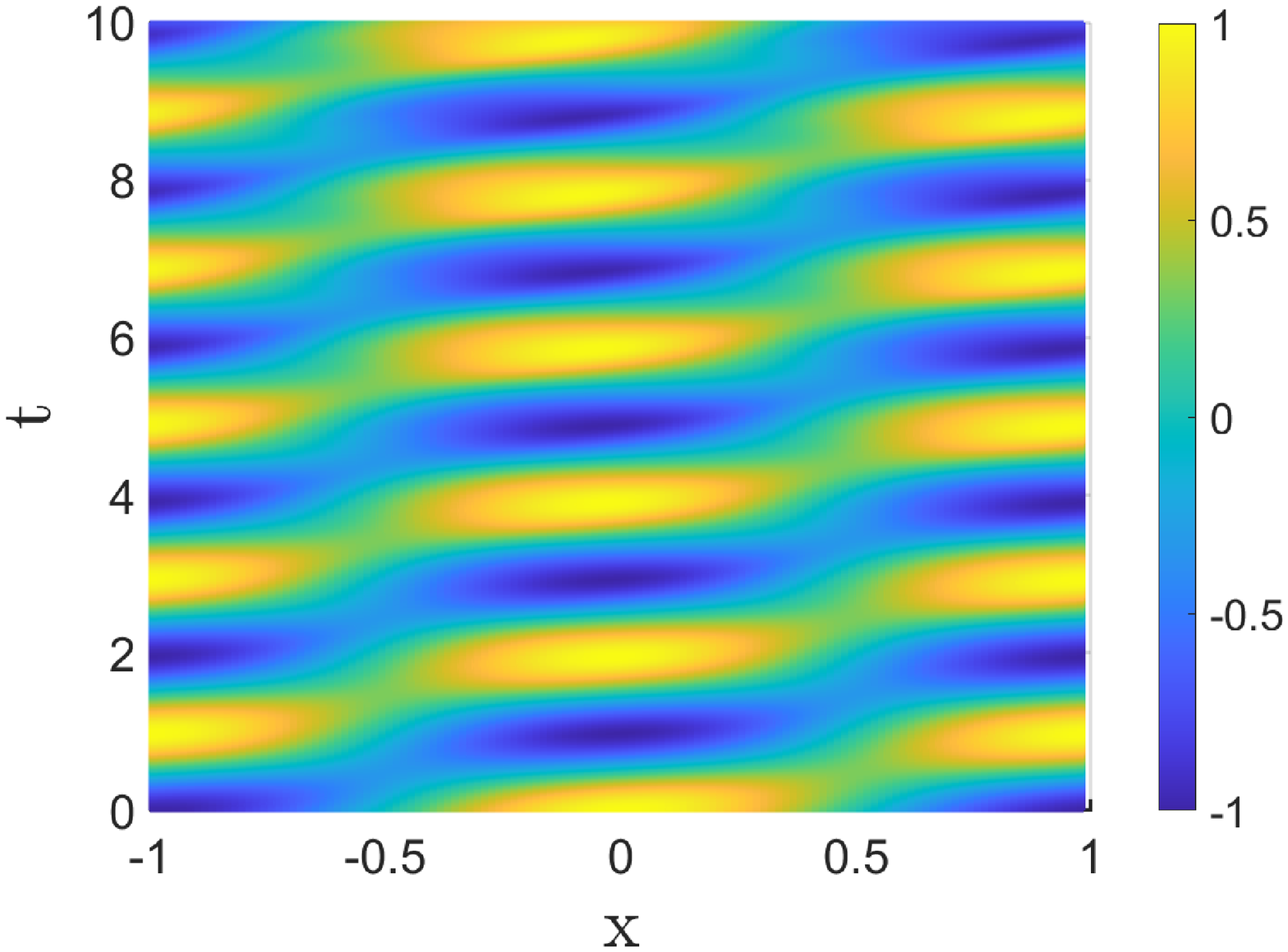}}~
\subfigure{\includegraphics[width=.33\textwidth, height=0.30\textwidth]{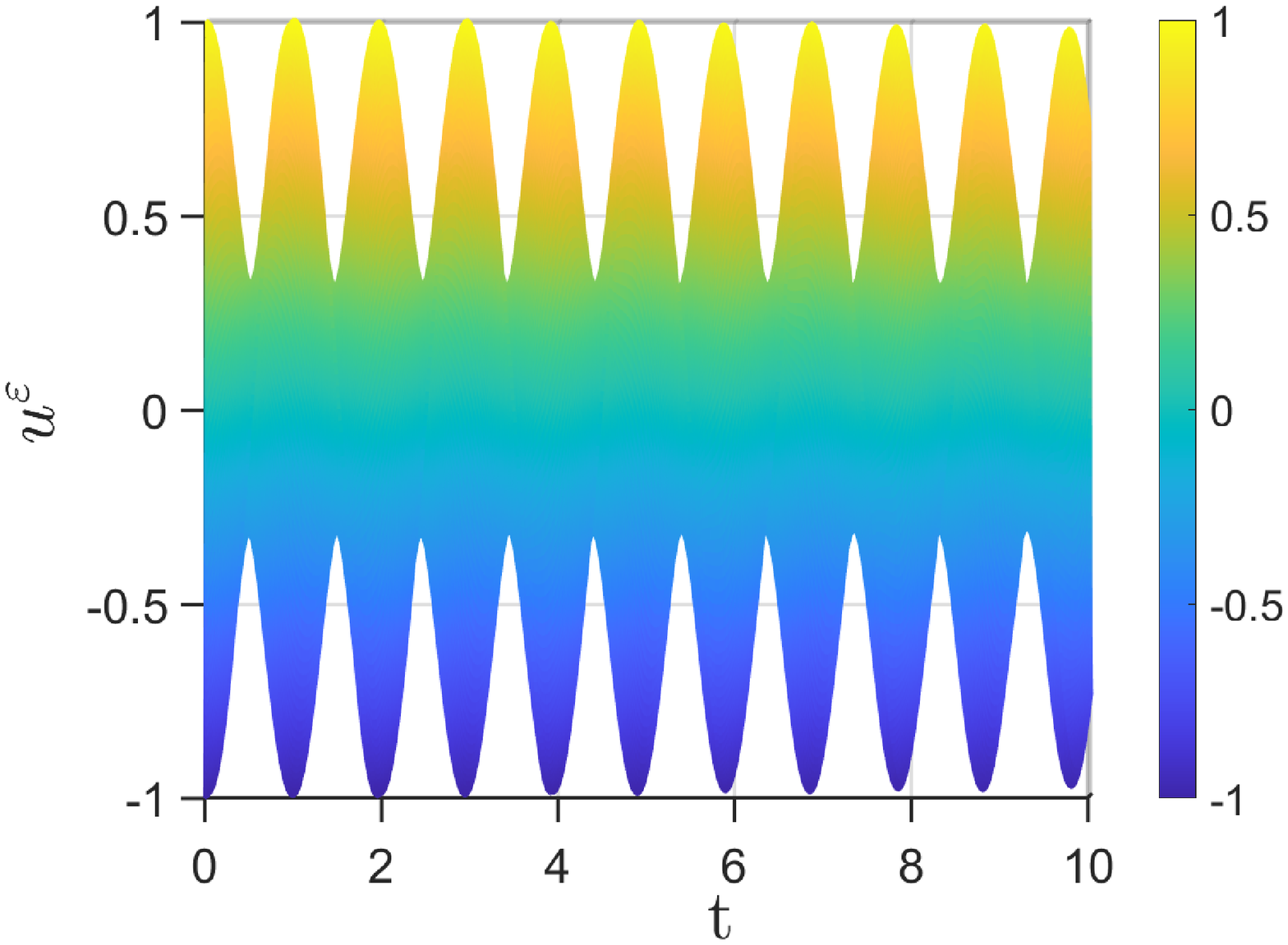}}~~
\subfigure{\includegraphics[width=.31\textwidth, height=0.30\textwidth]{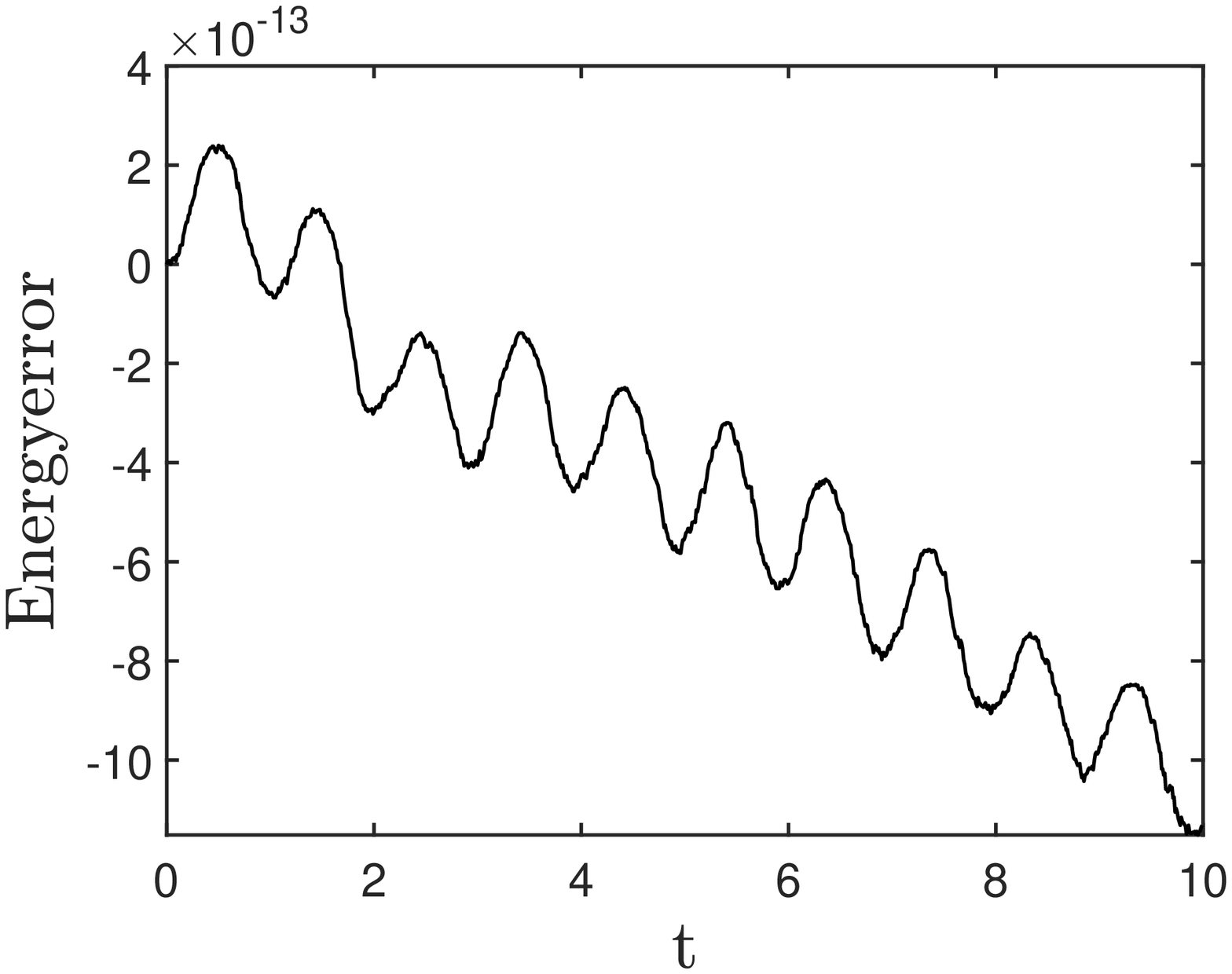}}
\caption{The evolution of the waveform (left), the shift of profile of $u^{\varepsilon}$ (middle), the energy error in $t\in [0,10]$ (right).}
\label{fig:case2_energyerror}
\end{figure}
\end{center}
\begin{center}
\begin{figure}[htbp]
\centering
\subfigure[t=0]{\includegraphics[width=.33\textwidth, height=0.30\textwidth]{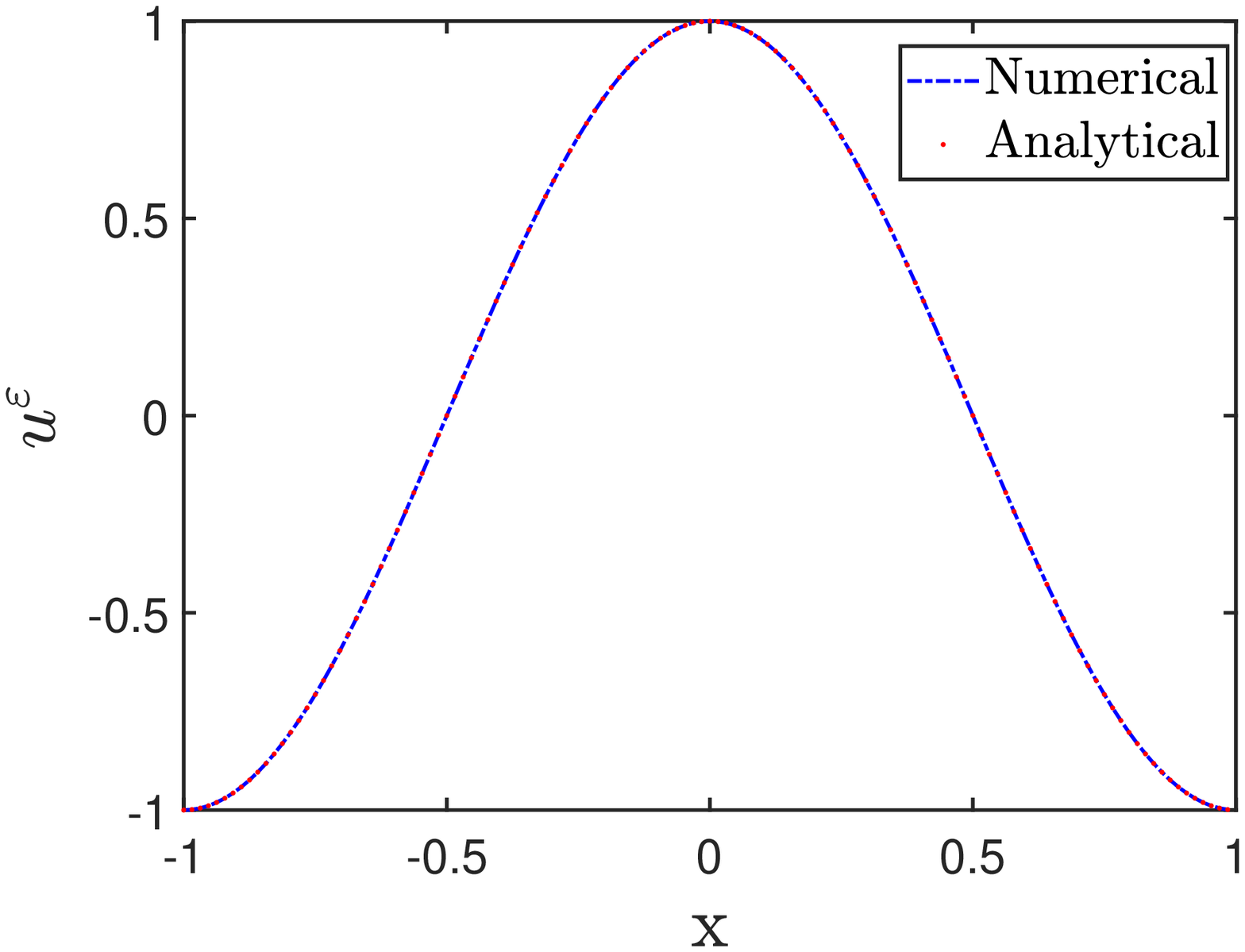}}~
\subfigure[t=1]{\includegraphics[width=.33\textwidth, height=0.30\textwidth]{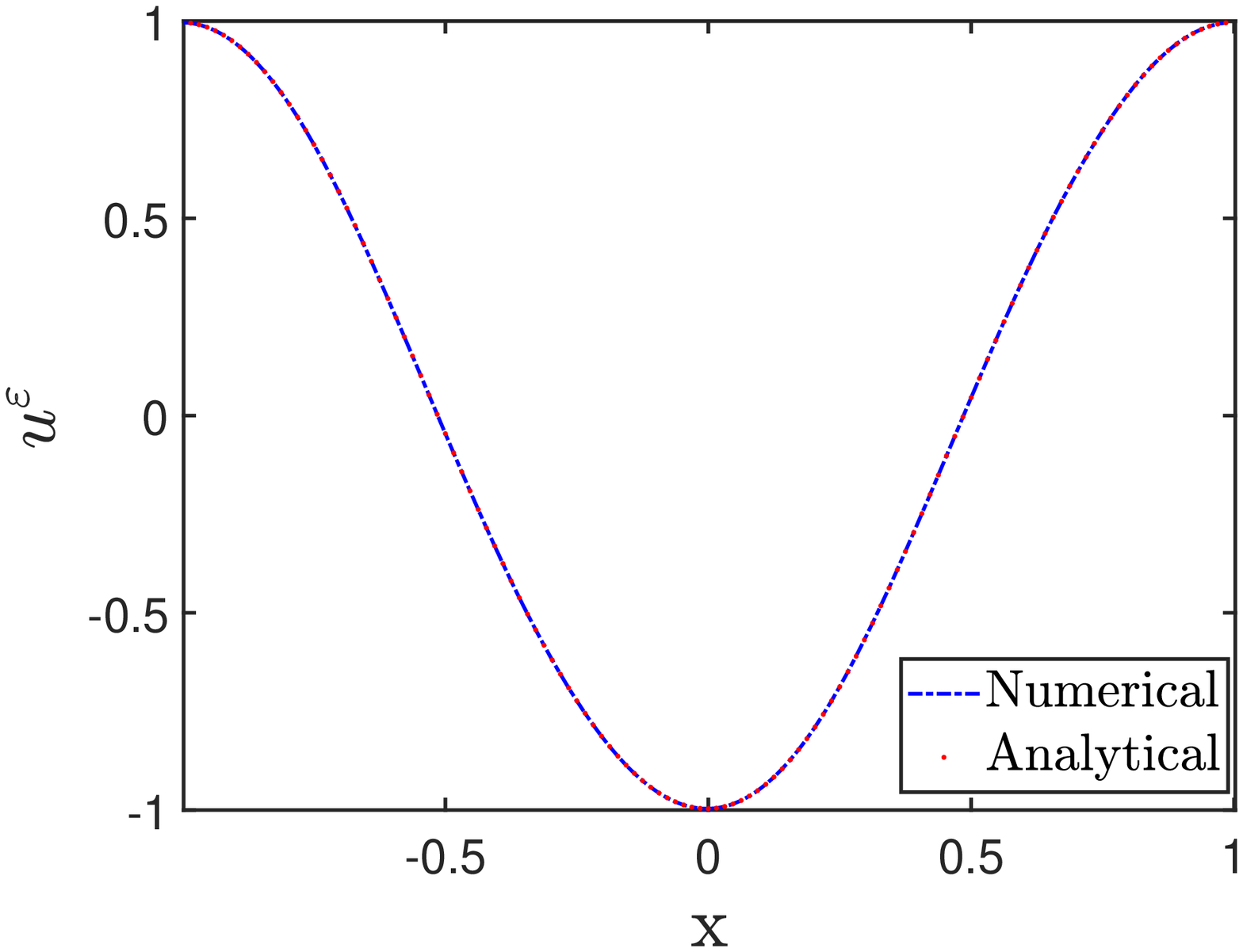}}~~
\subfigure[t=5]{\includegraphics[width=.31\textwidth, height=0.30\textwidth]{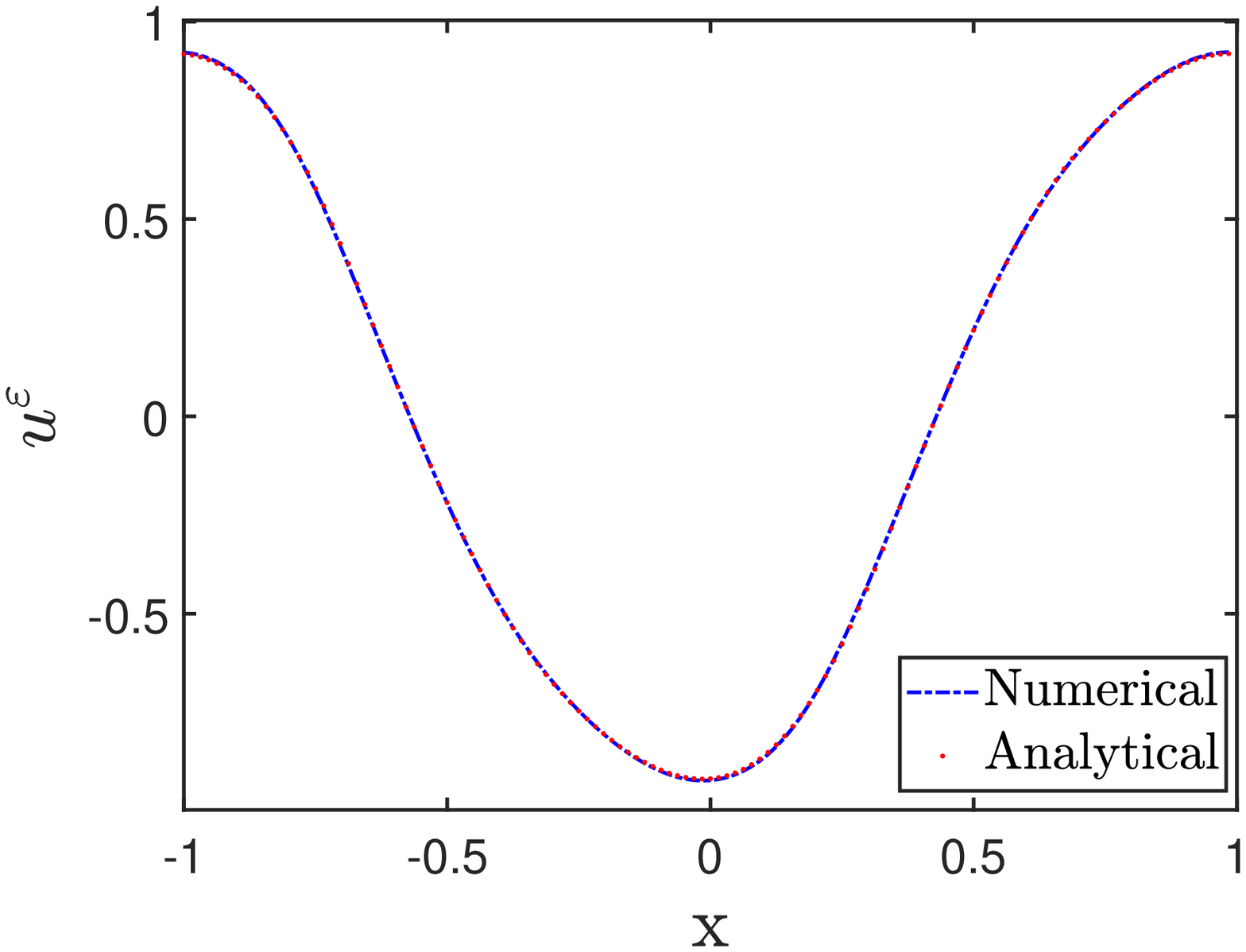}}
\caption{The numerical solution and analytical solution at three different times.}
\label{fig:case2_u}
\end{figure}
\end{center}

\section{Conclusions}\label{conclusion}
Two regularized energy-preserving finite difference methods: the CNFD (\ref{CNFD}) and the SIEFD (\ref{SIFD1}) are proposed and analyzed for the LogKGE (\ref{LogKGE}). Error estimates were rigorously estimated by utilizing the energy method, the cut-off technique and the inverse inequality, which showed that the FDTD methods at the order $O(h^{2}+\frac{\tau^{2}}{\varepsilon^{2}})$ in semi-$H^{1}$ norm. Besides, the error bounds were confirmed by some numerical results and the convergence order from the LogKGE to the RLogKGE is $O(\varepsilon)$. Based on the convergence, stability, energy conserving and computational results, we conclude that the CNFD and the SIEFD schemes are favorable for the LogKGE (\ref{LogKGE}).

\section*{Acknowledgments}
This work is supported by the National Natural Science Foundation of China (Grant No. 11971481,11901\\577), the Natural Science Foundation of Hunan (Grant No.S2017JJQNJJ0764,  S2020JJQNJJ1615), the Basic Research Foundation of National Numerical Wind Tunnel Project (No. NNW2018-ZT4A08). Research Fund of NUDT (Grand No. ZK17-03-27,ZK19-37), and the fund from Hunan Provincial Key Laboratory of Mathematical Modeling and Analysis in Engineering (Grand No.2018MMAEZD004).\section*{References}
\bibliographystyle{elsarticle-num}
\bibliography{logkg_energy}

\end{document}